\documentclass[11pt,a4paper]{amsart}

\usepackage{amsfonts,amsmath,amssymb,amsthm}
\usepackage[utf8]{inputenc}
\usepackage[T1]{fontenc}

\usepackage{hyperref}   
\usepackage{todonotes}
\usepackage{tikz}   
\usetikzlibrary{shapes,arrows,calc,matrix,backgrounds,positioning,decorations.pathmorphing,patterns,arrows.meta,cd}

\usepackage{float}  

\newcommand{\Rr}{\mathbb{R}}

\renewcommand{\le}{\leqslant}

\renewcommand{\epsilon}{\varepsilon}

\newcommand{\defi}[1]{\emph{#1}}

{\theoremstyle{plain}
\newtheorem{theorem}{Theorem}[section]    
\newtheorem{lemma}[theorem]{Lemma}       
\newtheorem{proposition}[theorem]{Proposition}      
\newtheorem{corollary}[theorem]{Corollary}      
}
{\theoremstyle{remark}

\newtheorem{definition}[theorem]{Definition}      
\newtheorem*{remark*}{Remark}  
\newtheorem*{question*}{Question}  
\newtheorem{remark}[theorem]{Remark}   
\newtheorem{example}[theorem]{Example}
}

\makeatletter
\def\subsection{\@startsection{subsection}{2}%
	\z@{.5\linespacing\@plus.7\linespacing}{.3\linespacing}%
	{\normalfont\bfseries}}
\makeatother

\newcommand{\doubletilde}[1]{\tilde{\raisebox{0pt}[0.85\height]{$\tilde{#1}$}}}

\usepackage[a4paper]{geometry}
\title{Poincar\'e--Reeb graphs of real algebraic domains}

\author{Arnaud Bodin}
\author{Patrick Popescu-Pampu}
\author{Miruna-\c Stefana Sorea}

\email{arnaud.bodin@univ-lille.fr}
\email{patrick.popescu-pampu@univ-lille.fr}
\email{mirunastefana.sorea@sissa.it}

\address{Universit\'e de Lille, CNRS, Laboratoire Paul Painlev\'e, 59000 Lille, France}
\address{Universit\'e de Lille, CNRS, Laboratoire Paul Painlev\'e, 59000 Lille, France}
\address{SISSA - Scuola Internazionale Superiore di Studi Avanzati, 34136 Trieste, Italy and  Lucian Blaga University of Sibiu, Romania}



\date{\today}

\begin{document}

\begin{abstract}	
An \emph{algebraic domain} is a closed topological subsurface 
of a real affine plane whose boundary consists of disjoint smooth connected
components of real algebraic plane curves. 
We study the geometric shape of an algebraic domain by collapsing all vertical segments
contained in it: this yields a \emph{Poincar\'e--Reeb graph}, which is naturally transversal to the foliation by vertical lines.
We show that any transversal graph whose vertices have only 
valencies $1$ and $3$ and are situated on distinct vertical lines can be realized as a Poincar\'e--Reeb graph.
\end{abstract}

\maketitle

\section{Introduction}

An \emph{algebraic domain} 
$\mathcal{D}$ is a closed subset of an affine plane, homeomorphic to a  
surface with boundary, whose boundary $\mathcal{C}$ is a union of disjoint smooth connected components of real algebraic plane curves. This paper is dedicated to the study of the geometric shape of algebraic domains. 

\medskip

\textbf{Context and previous work.}  
In \cite{sorea2018shapes,sorea2020measuring}, 
the third author studied the non-convexity of the disks $\mathcal{D}$ bounded by 
the connected components $\mathcal{C}$ of the levels of 
a real polynomial function $f(x,y)$ contained in sufficiently small neighborhoods 
of strict local minima.  The principle was to collapse to points the maximal vertical segments 
contained inside $\mathcal{D}$. This yielded a special type of 
tree embedded in a topological space homeomorphic to $\Rr^2$. It was called the 
 \emph{Poincar\'e--Reeb tree} associated to $\mathcal{C}$ and to the projection 
$(x,y) \mapsto x$, and it measured the non-convexity of $\mathcal{D}$. 
Conversely, given a tree $T$ of a special kind embedded in a plane, 
\cite[Theorem 3.34]{sorea2018shapes} 
presented a construction of a polynomial function $f(x,y)$ 
with a strict local minimum at $(0,0)$, whose Poincar\'e--Reeb tree near $(0,0)$ is $T$.

The terminology ``Poincar\'e--Reeb'' introduced in \cite[Definition 2.24]{sorea2018shapes} was inspired by a similar construction used in Morse theory, namely by the classical graph 
introduced by Poincar\'e in his study of $3$-manifolds \cite[1904, Fifth supplement, p.~221]{Po}, and rediscovered by Reeb \cite{Re}
in arbitrary dimension. Reeb graphs encode the topology of level sets of real-valued functions on manifolds. Reeb graphs appear as useful tools in the study of singularity theory of differentiable maps; see  \cite{saeki2022}, \cite{MR2828379}. For a survey with a view towards applications in computational topology and data visualization, we refer the reader to \cite{MR3822880} and references therein. Studies of more general Reeb spaces have been done in several recent works such as \cite{MR3505333}, \cite{MR3685708}, \cite{MR2504290}, \cite{MR4472570}. Some very recent work in this area are, for instance, \cite{MR4476514}, \cite{MR4548150}. Applications of Reeb graphs in nonparametric statistics and data analysis are presented for instance in \cite{MR3850451}.

\medskip

\textbf{Poincar\'e--Reeb graphs of real algebraic domains.}  
In this paper we extend the 
previous method of study of non-convexity  to algebraic domains $\mathcal{D}$ 
in $\Rr^2$.  When $\mathcal{D}$ is compact, the collapsing of maximal vertical segments 
contained in it yields a finite planar graph which is not necessarily a tree, called the 
\emph{Poincar\'e--Reeb graph} of $\mathcal{D}$ relative to the vertical direction. 
See Figure \ref{fig:PRgraphOfD} for a first idea of the definition. In it is represented 
also a section of the collapsing map above this graph, called a \emph{Poincar\'e--Reeb 
graph in the source}. It is well-defined up to isotopies stabilizing each vertical line. 
Such a section exists whenever the projection $x : \Rr^2 \to \Rr$ is in addition \emph{generic} relative 
to the boundary $\mathcal{C}$ of $\mathcal{D}$, that is, $\mathcal{C}$ has no vertical 
bitangencies, no vertical inflectional tangencies and no vertical asymptotes.

\begin{figure}[H]
	\begin{center}
		\small
		\tikzstyle{every picture}=[scale=1.0*0.75]
		\input{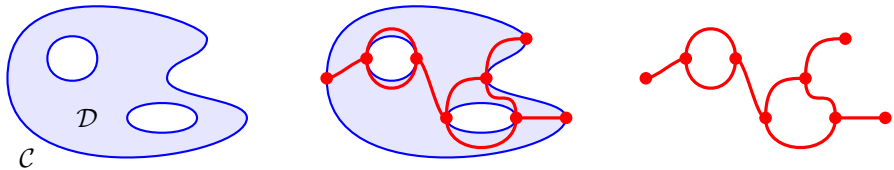}
	\end{center}	
  \caption{A Poincar\'e--Reeb graph: a curve $\mathcal{C}$ bounding a real algebraic domain 
     $\mathcal{D}$ (left); a Poincar\'e--Reeb graph in the source (center); 
     the Poincar\'e--Reeb graph (right).\label{fig:PRgraphOfD}}
\end{figure}

When $\mathcal{D}$ 
is non-compact but the projection $x : \Rr^2 \to \Rr$ is still proper in restriction to it, one gets 
an analogous graph, which has this time at least one unbounded edge. When the properness 
assumption on the projection is dropped but one assumes instead its genericity relative 
to  $\mathcal{C}$, then one may still define a Poincar\'e--Reeb graph in the source, again 
well-defined up to isotopies stabilizing the vertical lines.
Notice that the Poincar\'e--Reeb graph does not live in the same space as $\mathcal{D}$ even if the quotient space is homeomorphic to $\Rr^2$; we will work in the context of \emph{vertical planes} (see Definition \ref{def:vertplanes}) which is adapted for both the original plane and its quotient.

\medskip

\textbf{Finite type domains in vertical planes.}  
In order to be able to use our construction of 
Poincar\'e--Reeb graphs for the study of more general subsets of affine planes than algebraic domains, 
for  instance to topological surfaces bounded by semi-algebraic, piecewise smooth or even less 
regular curves, we give a purely topological description of the setting in which it may be 
applied. Namely, we define the notion of \emph{domain of finite type} 
$\mathcal{D}$ inside a \emph{vertical plane} $(\mathcal{P}, \pi)$: here $\pi : \mathcal{P} \to 
\Rr$ is a locally trivial fibration of an oriented topological surface $\mathcal{P}$ 
homeomorphic to $\Rr^2$ and $\mathcal{D}$ is a closed topological subsurface of 
$\mathcal{P}$, such that the restriction $\pi_{|_{\mathcal{D}}}$ is proper and the restriction 
$\pi_{|_{\mathcal{C}}}$ to the boundary $\mathcal{C}$ of $\mathcal{D}$ has a finite number of \emph{topological critical points}.

\medskip

\textbf{Main theorem.} 
Our main result is an answer in the generic case to the following 
question: \emph{given a transversal graph in a vertical plane $(\mathcal{P}, \pi)$, is it possible to find an algebraic domain whose Poincar\'e--Reeb graph is isomorphic to it?} Namely, we show 
that \textbf{each transversal graph whose vertices have valencies $1$ or $3$ and are situated 
on distinct levels of $\pi$ arises 
up to isomorphism from an algebraic domain in $\Rr^2$ such that the function 
$x : \Rr^2 \to \Rr$ is generic relative to it.} 
Our strategy of proof is to first realize the graph via a smooth function. 
Then we recall a Weierstrass-type theorem that approximates any smooth function 
by a polynomial function and we adapt its use in order to control vertical tangencies. 
In this way we realize any given generic compact transversal graph as the Poincar\'e--Reeb 
graph of a compact algebraic domain. Finally, we explain how to construct non-compact 
algebraic domains realizing some of the non-compact transversal graphs. Roughly speaking, 
we do this by adding branches to a compact curve.

\medskip

\textbf{Structure of the paper.}    Section \ref{sec:compactCase} is devoted to the definitions and 
several general properties of the notions vertical plane, finite type domain, Poincar\'e--Reeb graph, real algebraic domain and transversal graph in the compact setting. Section \ref{sec:algebraic} is  dedicated to the case where the real algebraic domain $\mathcal{D}$ is compact and connected. 
In it we present the main result of our paper, namely the algebraic realization of compact, connected, 
generic transversal graphs as Poincar\'e--Reeb graphs of connected algebraic domains (see Theorem \ref{th:realizationA}).
Section \ref{sec:noncompact} presents the case where $\mathcal{D}$ is non-compact and $\mathcal{C}$ is connected. Finally, in Section \ref{sec:general} we focus on the general situation, where $\mathcal{D}$ may be both non-compact and disconnected.

\medskip

\emph{Acknowledgements.}
We thank the referees for their useful remarks and references, which improved the quality of our manuscript. In particular, we are grateful 
for the simplification of the proof of Proposition \ref{th:graph}, using the Vietoris--Begel theorem \cite{MR87106} and for the references to the $\mathcal{C}^k$ Weierstrass approximation theorem. We also thank Antonio Lerario for referring us to  \cite{lerarioStecconi}, for a possible approach towards estimating the degree of the real algebraic domains.
This work was supported in part by the Labex CEMPI  (ANR-11-LABX-0007-01) 
and by the ANR project LISA (ANR-17-CE40-0023-01). The first author thanks the University of Vienna for his visit at the Wolfgang Pauli Institute. M.-\c S. Sorea would like to thank Antonio Lerario for the very supportive working environment during her postdoc at SISSA, in  Trieste, Italy.

\section{Poincar\'e--Reeb graphs of domains of finite type in vertical planes}  
\label{sec:compactCase}

\subsection{Algebraic domains}   

An \defi{affine plane} $\mathcal{P}$ 
is a principal homogeneous space under the action of a real vector 
 space of dimension $2$. 
 It  has a natural structure of real affine surface (the term ``affine'' being 
 taken now in the sense of algebraic geometry) and also a canonical compactification 
 into a real projective plane. Therefore, one may speak of real-valued polynomial functions 
 $f : \mathcal{P} \to \Rr$ as well as of algebraic curves in $\mathcal{P}$ of given degree. 
 We are interested in the following types of surfaces embedded in affine planes:

\begin{definition}
\label{def:algebraicDomain}
 An \defi{algebraic domain} is a closed subset $\mathcal{D}$ of an affine plane, 
 homeomorphic to a surface with boundary, whose boundary $\mathcal{C}$ 
 is a disjoint union of finitely many smooth connected components of real algebraic plane curves. 
\end{definition}

\begin{example}
Consider the algebraic curve $\overline{\mathcal{C}_1}$ of equation $(f_1(x,y)  = 0)$ 
with $f_1(x,y) = y^2 - (x-1)(x-2)(x-3)$ 
and $\overline{\mathcal{C}_2}$ of equation $(f_2(x,y)  = 0)$ with $f_2(x,y) = y^2 - x(x-4)(x-5)$.
Each of these curves has two connected components, a compact one (an \emph{oval} denoted 
by $\mathcal{C}_i$) and a non-compact one.
Let  $\mathcal{D}$ be  the ring surface bounded by $\mathcal{C}_1$ and $\mathcal{C}_2$. 
By definition, it is an algebraic domain. 
\end{example}

\begin{figure}[H]
    \centering
    \includegraphics[scale=0.5]{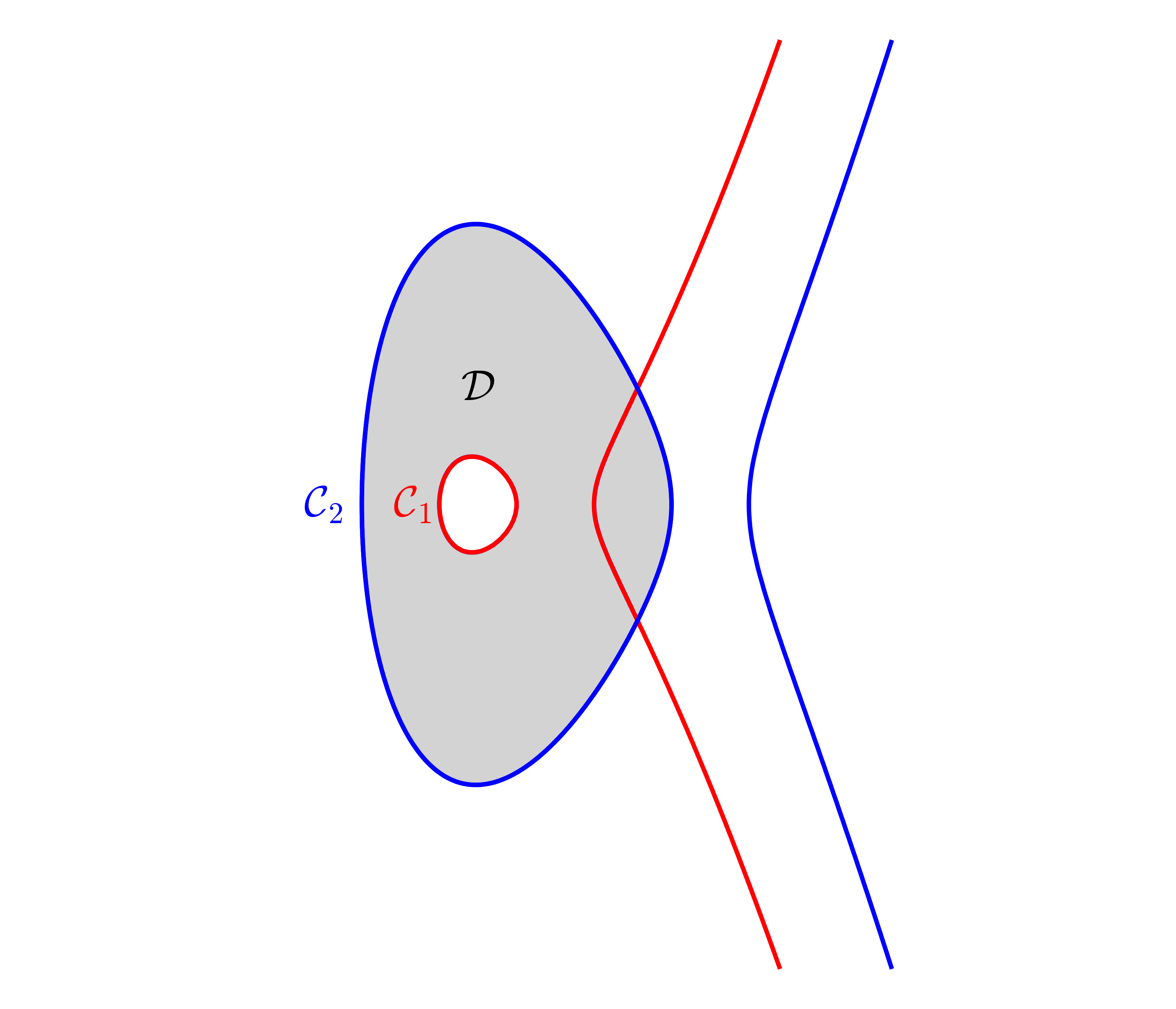}
    \caption{The algebraic domain $\mathcal{D}$ bounded by $\mathcal{C}_1$ and $\mathcal{C}_2$.}
    \label{fig:twoCurves}
\end{figure}

\subsection{Domains of finite type in vertical planes}   
\label{subsec:domfintype}

Assume that $\mathcal{D}$ is an algebraic domain in $\mathbb{R}^2$. 
We will study its non-convexity by collapsing  to points the maximal vertical segments contained 
inside $\mathcal{D}$ (see Definition \ref{def:equivRel} below). 
The image of $\mathbb{R}^2$ by such a collapsing map   
cannot be identified canonically to $\mathbb{R}^2$, and it has not even a canonical 
structure of affine plane. But in many cases it is homeomorphic to $\mathbb{R}^2$, it inherits from the 
starting affine plane $\Rr^2$  a canonical orientation and the function $x : \Rr^2 \to \Rr$ 
descends to it as a locally trivial topological fibration. This 
fact motivates the next definition:

\begin{definition}  \label{def:vertplanes}
   A \defi{vertical plane} is a pair $(\mathcal{P},\pi)$ such that $\mathcal{P}$ is a topological space    
   homeomorphic to $\mathbb{R}^2$, endowed with an orientation, and 
   $\pi:\mathcal{P}\rightarrow\mathbb{R}$ is a locally trivial  topological fibration. 
   The map $\pi$ is called the \defi{projection} of the 
    vertical plane and its fibers are called the \defi{vertical lines} of the vertical plane. 
    A vertical plane $(\mathcal{P},\pi)$ is called \defi{affine} if $\mathcal{P}$ is an affine plane 
    and $\pi$ is affine, that is, a polynomial function of degree one.
    The \defi{canonical affine vertical plane} is $(\Rr^2, x:\Rr^2\rightarrow\Rr)$. 
\end{definition}

Let $(\mathcal{P},\pi)$ be a vertical plane. 
As the projection $\pi$ is locally trivial over a contractible base, it is globally trivializable. 
This implies that $\mathcal{P}$ is homeomorphic to the Cartesian product $\mathbb{R} \times V$, 
where $V$ denotes any vertical line of $(\mathcal{P},\pi)$. 
The assumption that $\mathcal{P}$ is homeomorphic to 
$\mathbb{R}^2$ implies that the vertical lines are homeomorphic to $\mathbb{R}$. 
We will say that a subset of a vertical line of $(\mathcal{P},\pi)$ which is homeomorphic to a usual 
segment of $\Rr$ is a \defi{vertical segment}.

Given a curve in a vertical plane, we may distinguish special points of it:

\begin{definition} 
\label{def:topcrit}
    Let $(\mathcal{P},\pi)$ be a vertical plane and $\mathcal{C}$ a \defi{curve}  
     in it, that is, a closed subset of it  
    which is a topological submanifold of dimension one.  
    The \defi{topological critical set} $\Sigma_{\text{top}} (\mathcal{C})$ of  $\mathcal{C}$ 
    consists of the \defi{topological critical points} of the restriction $\pi_{|_{\mathcal{C}}}$, 
    which are  those points $p\in \mathcal{C}$ in whose neighborhoods the restriction 
    $\pi_{|_{\mathcal{C}}}$ is not a local homeomorhism onto its image. 
\end{definition}

\begin{figure}[H]
	\begin{center}
		\small
		\tikzstyle{every picture}=[scale=1.0*0.65]
		\input{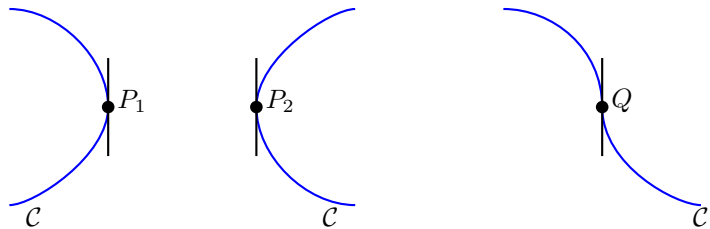}
	\end{center}	
	\caption{Two topological critical points $P_1$ and $P_2$ (which are critical points in the differential setting). The inflection point $Q$ is not a topological critical point but is a critical point in the differential setting.
	\label{fig:topcrit}}
\end{figure}

\begin{remark}   \label{rem:compcritpoints}
If $\mathcal{C}$ is an algebraic curve contained in an affine vertical plane, 
the topological critical set $\Sigma_{\text{top}} (\mathcal{C})$ is contained in the 
usual critical set $\Sigma_{\text{diff}}(\mathcal{C})$ of $\pi_{|_{\mathcal{C}}}$, 
but is not necessarily equal to it. For instance, any inflection point of $\mathcal{C}$ 
with vertical tangency and at which $\mathcal{C}$ crosses its tangent line belongs 
to $\Sigma_{\text{diff}}(\mathcal{C}) \setminus \Sigma_{\text{top}}(\mathcal{C})$ (see Figure \ref{fig:topcrit}).
\end{remark}

The topological critical set $\Sigma_{\text{top}} (\mathcal{C})$ is a closed subset of $\mathcal{C}$. 
In the neighborhood of an isolated topological critical point, the curve has a simple 
behavior:

\begin{lemma}  \label{lem:oneside}
    Let $(\mathcal{P},\pi)$ be a vertical plane and $\mathcal{C}$ a curve in it. Let 
    $p \in \mathcal{C}$ be an isolated topological critical point. 
    Then $\mathcal{C}$ lies locally on one side 
    of the vertical line passing through $p$. Moreover, there exists a neighborhood of $p$ 
    in $\mathcal{C}$, homeomorphic to a compact segment of $\Rr$, and such that 
    the restrictions of $\pi$ to both subsegments of it bounded by $p$ are 
    homeomorphisms onto their images.  
\end{lemma}

\begin{proof} 
Consider a compact arc $I$ of $\mathcal{C}$ whose interior is disjoint from 
$\Sigma_{\text{top}} (\mathcal{C})$. Identify it homeomorphically to a bounded  interval 
$[a,b]$ of $\Rr$. The  projection $\pi$ becomes a function $[a,b] \to \Rr$ 
devoid of topological critical points in $(a,b)$, that is, a strictly monotonic function. Consider 
now two such arcs $I_1$ and $I_2$ on both sides of $p$ in $\mathcal{C}$. 
The relative interior of their union $I_1 \cup I_2$ 
is a neighborhood with the stated properties. Moreover, $I_1 \cup I_2$ lies on only one side of the vertical line passing through $p$: 
otherwise $\pi$ would map $I_1$ homeomorphically to $[\alpha,x_0]$ and
$I_2$ homeomorphically to $[x_0,\beta]$, where $x_0=\pi(p)$ is a critical value and as $I_1$ and $I_2$ are on both side of the vertical line at $p$ we would have for instance $\alpha < x_0 < \beta$; this implies that $\pi : I_1 \cup I_2 \to [\alpha,\beta]$ is a homeomorphism, in contradiction with $p$ being a topological 
critical point of $\pi_{|_{\mathcal{C}}}$.
\end{proof}

As explained above, in this paper we are interested in the geometric shape 
of algebraic domains relative to a given ``vertical'' direction. 
But the way of studying them through the collapse of vertical segments may be extended 
to other kinds of subsets of real affine planes, for instance to topological surfaces bounded 
by semi-algebraic, piecewise-smooth or even less regular curves, provided 
they satisfy supplementary properties relative to the chosen projection. 
Definition \ref{def:domainFiniteType} below describes the most general context we could find 
in which the collapsing construction yields a new vertical plane and a finite graph in it, 
possibly unbounded. It is purely topological, involving no differentiability assumptions.

\begin{definition}\label{def:domainFiniteType}
Let $(\mathcal{P},\pi)$ be a vertical plane. Let $\mathcal{D}\subset\mathcal{P}$ be a closed subset 
homeomorphic to a surface with non-empty boundary. Denote by $\mathcal{C}$ its boundary. 
We say that $\mathcal{D}$ is a 
 \defi{domain of finite type} in $(\mathcal{P},\pi)$ if:
\begin{enumerate}
    \item \label{properness} 
         the restriction $\pi_{|_{\mathcal{D}}}: \mathcal{D}\to \Rr$  is proper;
    \item  \label{fincrit}
        the topological critical set $\Sigma_{\text{top}} (\mathcal{C})$ is finite.
\end{enumerate}
\end{definition}

\medskip

\begin{example}
~
\begin{figure}[H]
	\begin{center}
		\small
		\tikzstyle{every picture}=[scale=1.0*0.80]
		\input{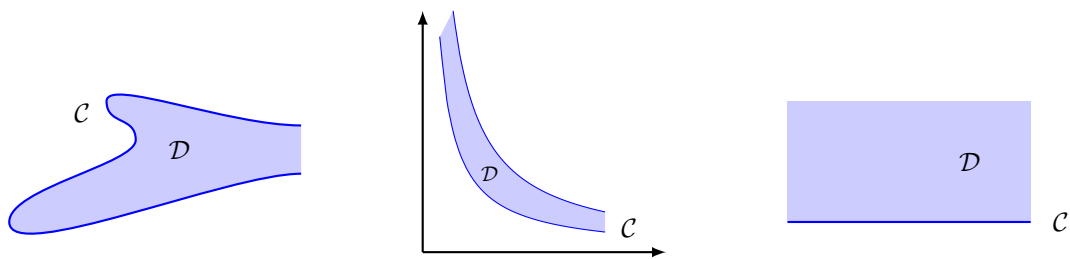}
	\end{center}	
	\caption{One example of a domain of finite type (left). Two examples of domains that are not of finite type (center and right).
	\label{fig:finitetype}}
\end{figure}

Condition (\ref{properness}) implies that the restriction $\pi_{|_{\mathcal{C}}}: \mathcal{C}\to \Rr$ 
is also proper, which means that $\mathcal{C}$ has no connected components which are 
vertical lines or have vertical asymptotes. For instance, consider an algebraic domain 
contained in the positive quadrant of 
the canonical vertical plane $\Rr^2$,  limited by two distinct level curves of the function $xy$ 
(see the middle drawing of Figure \ref{fig:finitetype}). 
It satisfies condition (\ref{fincrit}) as it has no topological critical points, but as 
$\mathcal{C}$  has a vertical asymptote (the $y$-axis),  
it does not satisfy condition (\ref{properness}), therefore it is not a domain of finite type. Note 
that condition (\ref{properness}) is stronger than the properness of $\pi_{|_{\mathcal{C}}}$. 
For instance, the upper half-plane in $(\Rr^2, x)$ does not satisfy condition (\ref{properness}), 
but $x_{|_{\mathcal{C}}}$ 
is proper for it (see the right drawing of Figure \ref{fig:finitetype}).
\end{example}

We distinguish two types of topological critical points on the boundaries of domains of finite type:  

\begin{definition}  \label{def:twotypes} 
      Let $(\mathcal{P},\pi)$ be a vertical plane and $\mathcal{D}\subset\mathcal{P}$ 
      a domain of finite type, whose boundary is denoted by  $\mathcal{C}$. 
      A topological critical point of $\mathcal{C}$ is called:
         \begin{itemize}
             \item an \defi{interior topological critical point of $\mathcal{D}$} 
                  if the vertical line passing through it  
                 lies locally inside $\mathcal{D}$;   
             \item an \defi{exterior topological critical point  of $\mathcal{D}$} 
                  if the vertical line passing through it 
                 lies locally outside $\mathcal{D}$.
         \end{itemize}
\end{definition}

\begin{figure}[H]
	\begin{center}
		\small
		\tikzstyle{every picture}=[scale=1.0*1]
		\input{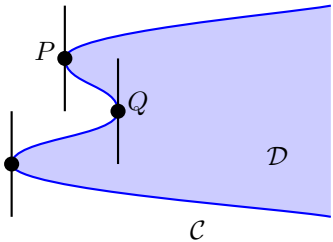}
	\end{center}	
    \caption{Example of an exterior topological critical point $P$ and an interior topological 
        critical point $Q$  of $\mathcal{D}$.}
    \label{fig:intExtCritPoints}
\end{figure}

One has the following consequence of Definition \ref{def:domainFiniteType}:

\begin{proposition}\label{prop:finiteSegments}
    Let $(\mathcal{P},\pi)$ be a vertical plane and $\mathcal{D}\subset\mathcal{P}$ 
      a domain of finite type. Denote by $\mathcal{C}$ its boundary.  Then:
        \begin{enumerate}
           \item each topological critical point of $\pi_{|_{\mathcal{C}}}$ 
               is either interior or exterior in the sense of Definition \ref{def:twotypes};
           \item the fibers of the restriction $\pi_{|_{\mathcal{D}}} : \mathcal{D} \to \Rr$ 
              are homeomorphic to finite disjoint unions of compact segments of $\mathbb{R}$; 
           \item the curve $\mathcal{C}$ has a finite number of connected components.
         \end{enumerate}
\end{proposition}

\begin{proof}  
~
\begin{enumerate}
    \item This follows directly from Lemma \ref{lem:oneside} and Definition \ref{def:twotypes}.
    \item Let us consider a point $x_0\in\mathbb{R}$. By Definition \ref{def:domainFiniteType} 
       (\ref{properness}), since the set $\{x_0\}$ is compact, we obtain that the fiber  
       $\pi_{|_{\mathcal{D}}}^{-1}(x_0)$ is compact. Let now $p$ be a point of this fiber. 
       By looking successively 
       at the cases where $ p \in \mathcal{D} \setminus \mathcal{C}$, 
       $p \in \mathcal{C} \setminus \Sigma_{\text{top}} (\mathcal{C})$, 
        $p$ is an interior and $p$ is an exterior topological critical point, we see that there exists a 
        compact vertical segment $K_p$,  neighborhood 
        of $p$ in the vertical line $\pi^{-1}(x_0)$, such that $\pi_{|_{\mathcal{D}}}^{-1}(x_0) \cap K_p$ 
        is a compact vertical segment. 
        
\begin{figure}[H]
	\begin{center}
		\small
		\tikzstyle{every picture}=[scale=1.0*0.8]
		\input{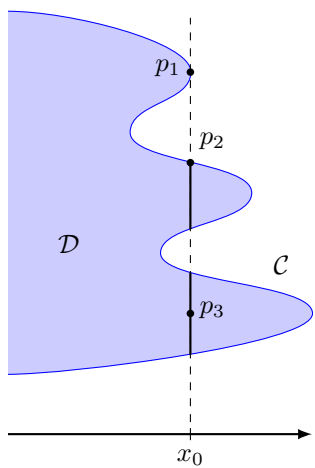}
	\end{center}	
	\caption{Different types of point $p$ in $\pi^{-1}(x_0) \cap \mathcal{D}$.}
	\label{fig:proofcritpoints}
\end{figure}        
        
        As $\pi_{|_{\mathcal{D}}}^{-1}(x_0)$ is compact, it may be 
        covered by a finite collection of such segments $K_p$. This implies that 
        $\pi_{|_{\mathcal{D}}}^{-1}(x_0)$ is a finite union of vertical segments (some of which 
        may be points).
    \item Let $\Delta_{\text{top}}(\mathcal{C}) \subset \Rr$ be the topological critical image of 
       $\pi |_{\mathcal{C}}$, that is, the image $\pi(\Sigma_{\text{top}} (\mathcal{C}))$ of the topological 
       critical set. As by Definition \ref{def:domainFiniteType}, $\Sigma_{\text{top}} (\mathcal{C})$ 
       is finite, $\Delta_{\text{top}}(\mathcal{C})$ is also finite. 
       Therefore, its complement $\Rr \setminus \Delta_{\text{top}}(\mathcal{C})$ 
       is a finite union of open intervals $I_i$. As $\pi_{|_\mathcal{D}}$ is proper, this is also the case 
       of $\pi_{|_\mathcal{C}}$. Therefore, for every such interval $I_i$ 
       the preimage $\pi_{|_\mathcal{C}}^{-1}(I_i)$ is a 
       finite union of arcs. This implies that $\mathcal{C}$ 
       is a finite union of arcs and points, therefore it has a finite number of connected 
       components.
\end{enumerate}
\end{proof}

\subsection{Collapsing vertical planes relative to domains of finite type}

Next definition formalizes the idea of collapsing the maximal vertical segments contained in 
a domain of finite type, mentioned at the beginning of Subsection \ref{subsec:domfintype}.

\begin{definition}   
\label{def:equivRel}
Consider a vertical plane $(\mathcal{P},\pi)$ and let $\mathcal{D}\subset\mathcal{P}$ 
be a domain of finite type. We say that two points $P$ and $Q$ of $\mathcal{P}$ are 
\defi{vertically equivalent relative to $\mathcal{D}$}, denoted $P\sim_{\mathcal{D}} Q$, 
if the following two conditions hold:
\begin{itemize}
    \item $P$ and $Q$ are on the same fiber of $\pi$, that is $\pi(P)=\pi(Q)=: x_0\in\Rr$;
    \item either the points $P$ and $Q$ are on the same connected component of $\pi^{-1}(x_0) \:  \cap \:  \mathcal{D}$, or $P= Q \notin \mathcal{D}$.
\end{itemize}
Denote by $\tilde{\mathcal{P}}$ the quotient $\mathcal{P}/{\sim_{\mathcal{D}}}$ of 
$\mathcal{P}$  by the vertical equivalence relation relative to $\mathcal{D}$. 
We call it the \defi{$\mathcal{D}$-collapse of $\mathcal{P}$}. 
The associated quotient map $\rho_{\mathcal{D}} : \mathcal{P} \to \tilde{\mathcal{P}}$ 
is called the \defi{collapsing map relative to $\mathcal{D}$}. 
\end{definition}

\begin{figure}[H]
	\begin{center}
		\small
		\tikzstyle{every picture}=[scale=1.0*0.8]
		\input{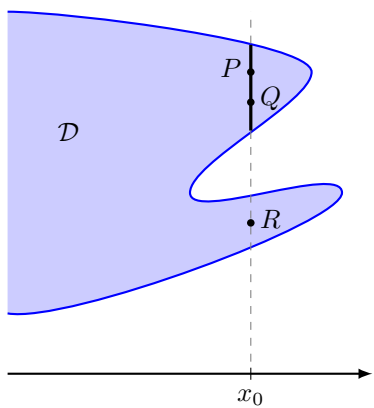}
	\end{center}	
	\caption{The points $P$ and $Q$ are vertically equivalent relative to $\mathcal{D}$: $P\sim_{\mathcal{D}} Q$. However, $P$ and $Q$ are  not equivalent to $R$.}
	\label{fig:defequivalence}
\end{figure} 
Next proposition shows that the $\mathcal{D}$-collapse of $\mathcal{P}$ is naturally a new  vertical plane, which is the reason why we introduced this notion in Definition \ref{def:vertplanes}.

\begin{proposition}    \label{prop:quotientIsAPlane}
  Let $(\mathcal{P},\pi)$ be a vertical plane and $\mathcal{D}$ be a domain of finite type in it. 
  Consider the collapsing map $\rho_{\mathcal{D}} : \mathcal{P} \to \tilde{\mathcal{P}}$ 
  relative to $\mathcal{D}$. Then: 
     \begin{itemize}
        \item $\tilde{\mathcal{P}}$ is homeomorphic to $\Rr^2$;
        \item the projection $\pi$ descends to a function $ \tilde{\pi}: \tilde{\mathcal{P}} \to \Rr$;
        \item $\rho_{\mathcal{D}}$ is a homeomorphism from $\mathcal{P} \setminus \mathcal{D}$ 
             onto its image;
        \item if one endows $\tilde{\mathcal{P}}$ with the orientation induced from that of $\mathcal{P}$ 
           by the previous homeomorphism, then 
           $(\tilde{\mathcal{P}},  \tilde{\pi} )$ is again a vertical plane, 
                and the following diagram is commutative: 
\begin{center}           
\begin{tikzcd}[column sep=small]
 	\mathcal{P} \arrow{rd}[swap]{\pi} \arrow{rr}{\rho_{\mathcal{D}}}    &  & \tilde{\mathcal{P}} \arrow{dl}{\tilde{\pi}} \\
 	 & \Rr & 
 \end{tikzcd}
\end{center}                    
      \end{itemize} 
\end{proposition}

The proof of Proposition \ref{prop:quotientIsAPlane} is similar to that of \cite[Proposition 4.3]{sorea2020measuring}.

\subsection{The Poincar\'e--Reeb graph of a domain of finite type}    

We introduce now the notion of Poincar\'e--Reeb set associated to a domain of 
finite type $\mathcal{D}$ in a vertical plane $(\mathcal{P},\pi)$. 
Whenever $\mathcal{P}$ is an affine plane 
and $\pi$ is an affine function, its role is to measure the non-convexity of
 $\mathcal{D}$ in the direction of the fibers of $\pi$.

\begin{definition}  \label{def:PRset}
  Let $(\mathcal{P},\pi)$ be a vertical plane and $\mathcal{D}\subset \mathcal{P}$ be a 
  domain of finite type. The \defi{Poincar\'e--Reeb set} of $\mathcal{D}$ 
  is the quotient ${\tilde{\mathcal{D}}} := \mathcal{D}/{\sim_{\mathcal{D}}}$, 
  seen as a subset of the  $\mathcal{D}$-collapse $\tilde{\mathcal{P}}$ of $\mathcal{P}$ 
  in the sense of Definition \ref{def:equivRel}. 
\end{definition}

The Poincar\'e--Reeb set from Definition \ref{def:PRset} has a canonical structure of 
graph embedded 
in the vertical plane $(\tilde{\mathcal{P}},\tilde{\pi})$, a fact which may be proved similarly to  \cite[Theorem 4.6]{sorea2020measuring}. Let us explain first how to get the vertices 
and the edges of $\tilde{\mathcal{D}}$.

\begin{definition}\label{def:verticesEdgesPR}
   Let $\mathcal{D}$ be a domain of finite type in a vertical plane  $(\mathcal{P},\pi)$, 
   and let $\mathcal{C}$ be its boundary. 
   A \defi{vertex} of the Poincar\'e--Reeb set $\tilde{\mathcal{D}}$ 
   is an element of $\rho_{\mathcal{D}}\left(\Sigma_{\text{top}} (\mathcal{C})\right)$.  
   A \defi{critical segment} of $\mathcal{D}$ is a connected component of a fiber 
   of $\pi_{|_{\mathcal{D}}}$ containing at least one element of $\Sigma_{\text{top}} (\mathcal{C})$.
   The \defi{bands} of $\mathcal{D}$ are the closures of the connected 
   components of the complement in $\mathcal{D}$ of the union of critical segments. 
   An \defi{edge} of $\tilde{\mathcal{D}}$ is the image  $\rho_{\mathcal{D}}(R)$ of a band  
   $R$  of $\mathcal{D}$ (see Figure \ref{fig206}).
\end{definition}

\begin{figure}[H]
	\begin{center}
		\small
		\tikzstyle{every picture}=[scale=1.0*0.7]
		\input{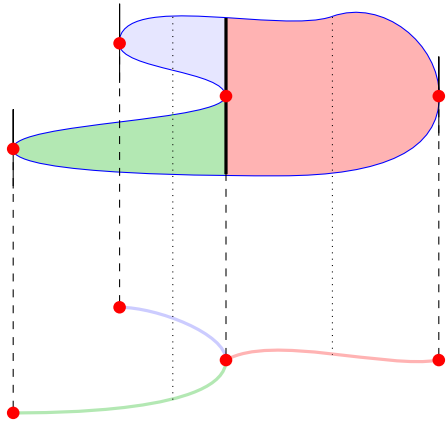}
	\end{center}	
    \caption{Construction of a Poincar\'e--Reeb set.  There are three bands, delimited by four critical segments (three of them are reduced to points). The interior of each edge of the graph is drawn in the same color as the corresponding band.} 
	\label{fig206}
\end{figure}

Each critical segment is either an exterior topological critical point 
in the sense of Definition  \ref{def:twotypes}  or a non-trivial 
segment containing a finite number of interior topological critical points in its interior 
(see Figure \ref{fig:twoInteriorPts} for an example with two such points). 

\begin{figure}[H]
	\begin{center}
		\small
		\tikzstyle{every picture}=[scale=1.0*0.8]
		\input{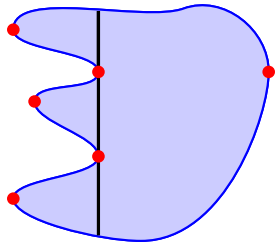}
	\end{center}
    \caption{A critical segment containing two interior topological critical points.} 
	\label{fig:twoInteriorPts}
\end{figure}

Next definition, motivated by Proposition \ref{prop:charcangraphPR} below,  introduces a special type of subgraphs of vertical planes:

\begin{definition}   
\label{def:transversalGraph} 
       Let $(\mathcal{P},\pi)$ be a vertical plane. A \defi{transversal graph} in $(\mathcal{P},\pi)$ 
       is a closed subset $G$ of $\mathcal{P}$ partitioned into finitely many points 
       called \defi{vertices} and subsets homeomorphic to open segments of $\Rr$ called 
       \defi{open edges}, such that:
          \begin{enumerate}
              \item each \defi{edge}, that is, the closure $\overline{E}$ of an open edge $E$, 
                  is homeomorphic to a closed segment of $\Rr$ and 
                  $\overline{E} \setminus E$ consists of $0$, $1$ or $2$ vertices;
              \item \label{condtransv} 
                   the edges are \defi{topologically transversal} to the vertical lines, that is, 
                   the restriction of $\pi$ to each edge is a homeomorphism onto its image in $\Rr$;
              \item \label{condproperedge}
                  the restriction $\pi_{|_{G}} : G \to \Rr$ is proper.
          \end{enumerate}
       A transversal graph is called \defi{generic} if its vertices are of valency $1$ or $3$ 
       and if distinct vertices lie on distinct vertical lines.
\end{definition}

Any transversal graph is homeomorphic to the complement of a subset of the set of vertices of valency $1$ inside a usual finite (compact) graph. This is due to the fact that some of its edges may be unbounded, in either 
one or both directions. Condition (\ref{condproperedge})
 from Definition \ref{def:transversalGraph} 
avoids $G$ having unbounded edges which are asymptotic to a vertical line of $\pi$. Note that 
we allow $G$ to be disconnected and the set of vertices to be empty. In this last case, $G$ is 
a finite union of pairwise disjoint open edges, each of them being sent by $\pi$ homeomorphically 
onto $\Rr$.

Here is the announced description of the canonical graph structure of 
the Poincar\'e--Reeb sets of domains of finite type in vertical planes:

\begin{proposition}   
\label{prop:charcangraphPR}
Let $\mathcal{D}$ be a domain of finite type in a vertical plane  $(\mathcal{P},\pi)$. 
Then each edge of the Poincar\'e--Reeb set 
$\tilde{\mathcal{D}}$ in the sense of Definition \ref{def:verticesEdgesPR} 
is homeomorphic to a closed segment of $\Rr$. Endowed with its vertices and edges, 
$\tilde{\mathcal{D}}$ is a transversal graph in $(\tilde{\mathcal{P}}, \tilde{\pi})$, without 
vertices of valency $2$. 
\end{proposition} 

The proof is straightforward using Proposition \ref{prop:finiteSegments}.
For an example, see the graph of Figure \ref{fig206}.

\medskip

Proposition \ref{prop:charcangraphPR} allows to give the following definition:

\begin{definition}   
\label{def:PRgraph}
 Let $\mathcal{D}$ be a domain of finite type in a vertical plane  $(\mathcal{P},\pi)$. 
 Its \defi{Poincar\'e--Reeb graph} is the Poincar\'e--Reeb set $\tilde{\mathcal{D}}$ 
 seen as a transversal graph in the $\mathcal{D}$-collapse 
 $(\tilde{\mathcal{P}}, \tilde{\pi})$ of $\mathcal{P}$ in the sense of Definition \ref{def:equivRel}, 
 when one endows it with vertices and edges in the sense of Definition \ref{def:verticesEdgesPR}. 
\end{definition}

The next result explains in which case the Poincar\'e--Reeb graph of a domain of finite type is generic 
in the sense of Definition \ref{def:transversalGraph}:

\begin{proposition}   
\label{prop:whenisPRgeneric}
Let $\mathcal{D}$ be a domain of finite type in a vertical plane  $(\mathcal{P},\pi)$. 
Denote by $\mathcal{C}$ its boundary. 
Then the Poincar\'e--Reeb graph $\tilde{\mathcal{D}}$ is a generic transversal graph in 
$(\tilde{\mathcal{P}}, \tilde{\pi})$ if and only if no two topological critical points 
of $\mathcal{C}$ lie on the same vertical line. 
\end{proposition}

\begin{proof}
This follows from Definition \ref{def:transversalGraph}, Definition \ref{def:twotypes} and Proposition \ref{prop:finiteSegments} (3). Vertices of valency $1$ of the Poincar\'e--Reeb graph correspond to exterior topological critical points, whereas vertices of valency $3$ correspond to interior topological critical points. 
\end{proof}

This proposition motivates:
\begin{definition}   
\label{def:gendom}
A domain of finite type in a vertical plane  is called \defi{generic} if no two topological critical points of its boundary lie on the same vertical line.
\end{definition}

Below we will define a related notion of generic direction with respect to an algebraic domain 
(see Definition \ref{def:genericDir}). For algebraic domains of finite type, up to a small rotation the vertical direction is generic, see Remark \ref{rem:MorseFunction} below. In other words, for all but a finite number of directions the projection is generic.

\subsection{Algebraic domains of finite type}    

Let us consider algebraic domains in the canonical affine vertical plane $(\Rr^2, x)$ 
(see Definition \ref{def:vertplanes}), 
in the sense of Definition \ref{def:algebraicDomain}. Not all 
of them are domains of finite type. For instance, the closed half-planes 
or the surface bounded by the hyperbolas $(xy = 1)$ and  $(xy = -1)$
are not of finite type, because the restriction of the projection $x$ to the domain is not proper.
Next proposition shows that this properness characterizes 
the algebraic domains which are of finite type, and that it may be checked simply:

\begin{proposition}  
\label{prop:whenalgdomfintype}
Let $(\mathcal{P},\pi)$ be an affine vertical plane and let $\mathcal{D}$ be an algebraic domain in it. Then the following conditions are equivalent:
\begin{enumerate}
     \item $\mathcal{D}$ is a domain of finite type.
     \item The restriction $\pi_{|_{\mathcal{D}}}: \mathcal{D} \to \Rr$ is proper.
     \item One fiber of $\pi_{|_{\mathcal{D}}}: \mathcal{D} \to \Rr$ is compact 
        and the boundary $\mathcal{C}$ of $\mathcal{D}$ does not contain vertical lines 
        and does not possess vertical asymptotes. 
 \end{enumerate}
\end{proposition}

\begin{proof} \emph{Let us prove first the implication (2) $\Rightarrow$ (1).} It is enough to show that 
  $\Sigma_{\text{top}} (\mathcal{C})$ is a finite set. 
The properness of $\pi_{|_{\mathcal{D}}}$ shows that $\mathcal{C}$ contains no vertical line. 
The set  of topological critical points being included in the set $\Sigma_{\text{diff}} (\mathcal{C})$ 
of differentiable critical points of $\pi|_{\mathcal{C}}$, it is enough to prove that this last set 
is finite. Consider a connected component $\mathcal{C}_i$ of $\mathcal{C}$ and its Zariski closure 
$\overline{\mathcal{C}_i}$ in $\mathcal{P}$. Let  $\mathcal{P}_{\pi}(\overline{\mathcal{C}_i})$  
be  its polar curve relative to $\pi$ (see \cite[Definition 2.43]{sorea2018shapes}). It is again 
an algebraic curve in $\mathcal{P}$, of degree smaller than the irreducible algebraic curve 
$\overline{\mathcal{C}}_i$. Therefore, the set 
$\overline{\mathcal{C}_i} \cap \mathcal{P}_{\pi}(\overline{\mathcal{C}_i})$ is finite, by Bézout's theorem. But this set 
contains $\mathcal{C}_i \cap \Sigma_{\text{diff}} (\mathcal{C}) $, which shows that $\pi|_{\mathcal{C}}$ 
has a finite number of differentiable critical points on each connected component $\mathcal{C}_i$. 
As $\mathcal{C}$ has a finite number of such components, we get that 
$\Sigma_{\text{diff}} (\mathcal{C})$ is indeed finite.

\emph{Let us prove now that (1) $\Rightarrow$ (3).} 
   Since $\mathcal{C}\subset\mathcal{D}$, we have by the properness condition of Definition \ref{def:domainFiniteType} (\ref{properness}) that  $\mathcal{C}$ does not contain vertical lines. Moreover, 
   if the boundary $\mathcal{C}$ of $\mathcal{D}$  possessed a vertical asymptote, 
   then we would obtain a contradiction with Definition \ref{def:domainFiniteType} 
   (\ref{properness}). Finally, since $\pi_{|_{\mathcal{D}}}$ is proper, 
   each of its fibers is compact.

\emph{Finally we prove that (3) $\Rightarrow$ (2).} 
   Since the boundary $\mathcal{C}$ of $\mathcal{D}$ does not contain vertical lines 
   and does not possess vertical asymptotes, the restriction $\pi |_{\mathcal{C}}$ is 
   proper. Moreover, it has a finite number of differentiable critical points, as the above proof 
   of this fact used only the absence of vertical lines among the connected 
   components of $\mathcal{C}$. We argue now similarly to our proof of 
   Proposition \ref{prop:finiteSegments} (3), 
   by subdividing $\Rr$ using the points of the topological critical image 
   $\Sigma_{\text{top}} (\mathcal{C})$. This set is finite, therefore $\Rr$ gets subdivided 
   into finitely many closed intervals. Above each one of them, $\mathcal{C}$ consists 
   of finitely many transversal arcs. If one fiber of $\pi_{|_{\mathcal{D}}}$ above such an 
   interval $I_j$ is compact, it means that $\pi_{|_{\mathcal{D}}}^{-1}(I_j)$ is a finite 
   union of bands bounded by pairs of such transversal arcs and compact vertical segments, 
   therefore $\pi_{|_{\mathcal{D}}}$ 
   is proper above $I_j$. In particular, its fibers above the extremities of $I_j$ are also 
   compact. In this way we show by progressive propagation from each interval with 
   a compact fiber to its neighbors, that $\pi_{|_{\mathcal{D}}}$ is proper above each 
   interval of the subdivision of $\Rr$ using $\Sigma_{\text{top}} (\mathcal{C})$. This implies 
   the properness of $\pi_{|_{\mathcal{D}}}$.
\end{proof}

Let us explain now a notion of \emph{genericity} of an affine function on an affine plane relative 
to an algebraic domain:

\begin{definition}    
\label{def:genericDir}
Let $\mathcal{D}$ be an algebraic domain in an affine vertical plane $(\mathcal{P},\pi)$, 
and let $\mathcal{C}$ be its boundary. The projection 
$\pi$ is called  \defi{generic with respect to $\mathcal{D}$} if 
$\mathcal{C}$ does not contain vertical lines and does not possess vertical asymptotes,  
vertical inflectional tangent lines and vertical multitangent lines (that is, vertical lines 
tangent to $\mathcal{C}$ at least at two points, or to a point of multiplicity greater than two). 
\end{definition}

\begin{remark}
\label{rem:MorseFunction}
Let $\mathcal{D}$ be an algebraic domain in an affine plane $\mathcal{P}$. 
Except for a finite number of directions of their fibers,  
all affine projections are generic with respect to $\mathcal{D}$ (see \cite[Theorem 2.13]{sorea2019permutations}). 
Note that the affine projection $\pi$ is generic with respect to $\mathcal{D}$ 
if and only if the restriction of $\pi$ to $\mathcal{C}$ is a proper excellent Morse function, 
i.e. all the critical points of $\pi_{|\mathcal{C}}$ are of Morse type and are situated 
on different level sets of $\pi_{|\mathcal{C}}$. Note also that if the algebraic domain 
$\mathcal{D}$ is moreover of finite type and $\pi$ is generic with respect to it in the sense 
of Definition \ref{def:genericDir}, then $\mathcal{D}$ is generic in the sense of Definition 
\ref{def:gendom}.
\end{remark}

\begin{proposition}  
\label{prop:gendomgengraph}
Let $\mathcal{D}$ be an algebraic domain of finite type in an affine vertical plane $(\mathcal{P},\pi)$. 
Assume that $\pi$ is generic with respect to $\mathcal{D}$ in the sense of Definition 
\ref{def:genericDir}. Then its Poincar\'e--Reeb graph 
is generic in the sense of Definition \ref{def:transversalGraph}. 
\end{proposition}

\begin{proof}
This is a consequence of   Proposition \ref{prop:whenisPRgeneric} and 
Remark \ref{rem:MorseFunction}, since by Definition \ref{def:topcrit}, the topological  
critical points of $\mathcal{C}$ are among the differential  critical points of the 
vertical projection $\pi |_{\mathcal{C}}$.
\end{proof}

\subsection{The invariance of the Euler characteristic}    

In this section we consider only \emph{compact} domains of finite type. This implies that their  
boundaries are also compact (see Figure \ref{fig:PRgraphOfD} for an example). 
Next result implies that the Betti numbers of the domain and of its Poincar\'e--Reeb graph 
are the same: 

\begin{proposition}   
\label{th:graph}
Let $\mathcal{D}$ be a compact domain of finite type in a vertical plane. Then 
$\mathcal{D}$ and its Poincar\'e--Reeb graph $\tilde{\mathcal{D}}$ are homotopically equivalent. In particular they have the 
same number of connected components and the same Euler characteristic. 
\end{proposition}

\begin{proof}
\emph{Connected components.} The collapsing map 
$\rho_{\mathcal{D}}$ of Definition \ref{def:equivRel}  
being continuous, each connected component of $\mathcal{D}$ 
is sent by $\rho_{\mathcal{D}}$ to a connected subset of $\tilde{\mathcal{D}}$. 
Those subsets are compact, as images of compact sets by a continuous map. 
They are moreover pairwise disjoint, by Definition \ref{def:equivRel} 
of the vertical equivalence relation relative to $\mathcal{D}$. Therefore, they are 
precisely the connected components of $\tilde{\mathcal{D}}$, which shows that 
$\rho_{\mathcal{D}}$ establishes a bijection between the connected components of 
$\mathcal{D}$ and $\tilde{\mathcal{D}}$. 

\emph{Homotopy equivalence.} 
We now may assume the $\mathcal{D}$ is connected.
By definition, for any $p \in \tilde{\mathcal{D}}$, $\rho_{\mathcal{D}}^{-1}(p)$ is an interval, then the Vietoris--Begle theorem, as stated by Smale in \cite{MR87106}, proves that $\rho_{\mathcal{D}} : \mathcal{D} \to \tilde{\mathcal{D}}$ induces an isomorphism for the corresponding homotopy groups. By the Whitehead theorem (see \cite[Theorem 4.5]{MR1867354}), we get
a homotopy equivalence between $\mathcal{D}$ and $\tilde{\mathcal{D}}$.
\end{proof}

Note that in Section \ref{sec:general} we will focus on the topology of the boundary curve 
$\mathcal{C}$ of $\mathcal{D}$, in terms of Betti numbers (see Proposition \ref{prop:bordBetti}).
The case where $\mathcal{D}$ is a disk was considered by the third author in her study of 
asymptotic shapes of level curves of polynomial functions $f(x,y) \in \Rr[x,y]$ 
near a local extremum (see \cite{sorea2018shapes,sorea2020measuring}). 

A direct consequence of Proposition \ref{th:graph} is:

\begin{proposition}   
\label{prop:treeCompact}
 If $\mathcal{D}\subset (\mathcal{P},\pi)$ is (homeomorphic to) a disk, then the 
 Poincar\'e--Reeb graph $\tilde{\mathcal{D}}$ of $\mathcal{D}$ is a tree. 
\end{proposition}

\begin{proof}
Proposition \ref{th:graph} implies that $\tilde{\mathcal{D}}$ is connected and that 
$\chi(\tilde{\mathcal{D}}) = 1$. But these two facts characterize the trees among the 
finite graphs. 
\end{proof}

If the disk $\mathcal{D}\subset (\mathcal{P},\pi)$ is an algebraic domain in a vertical affine plane 
and the projection $\pi$ is generic with respect to $\mathcal{D}$ in the sense of 
Definition \ref{def:genericDir}, then Proposition \ref{prop:gendomgengraph} implies that  
the Poincar\'e--Reeb graph $\tilde{\mathcal{D}}$ is a \defi{complete binary tree}: 
each vertex is either of valency $3$ (we call it then {\em interior}) 
or of valency $1$ (we call it then {\em exterior}).

\begin{figure}[H]
	\begin{center}
		\small
		\tikzstyle{every picture}=[scale=1.0*0.8]
		\input{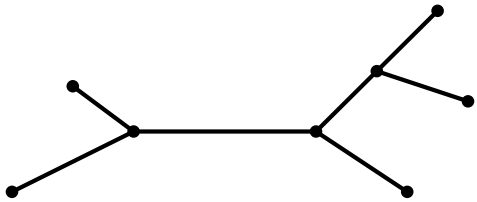}
	\end{center}	
	\caption{The Poincar\'e--Reeb graph of a disk relative to a generic projection} 
	is a complete binary tree.
	\label{fig0102}
\end{figure}

\subsection{Poincar\'e--Reeb graphs in the source}   \label{ssec:PRsource}      
 
Definition \ref{def:PRgraph} of the Poincar\'e--Reeb graph $\tilde{\mathcal{D}}$ 
of a finite type domain $\mathcal{D}$  in a vertical 
plane $(\mathcal{P}, \pi)$  is canonical. However, it yields a graph embedded in a  new vertical plane 
$\tilde{\mathcal{P}}$, which cannot be identified canonically to the starting one. 
When the Poincar\'e--Reeb graph is \emph{generic} in the sense of 
Definition \ref{def:transversalGraph}, it is possible to lift it to the starting plane.

\begin{proposition}  
	\label{prop:existsect}
    Let $\mathcal{D}$ be a finite type domain in a vertical plane $(\mathcal{P}, \pi)$. 
    If the Poincar\'e--Reeb graph $\tilde{\mathcal{D}}$ is generic, then the map 
    $(\rho_{\mathcal{D}})_{|_{\mathcal{D}}} : \mathcal{D} \to \tilde{\mathcal{D}}$ 
    admits a section, which is well defined up to isotopies stabilizing each vertical line.
\end{proposition}

\begin{proof}
    The genericity assumption means that above each vertex of $\tilde{\mathcal{D}}$ there 
    is a unique topological critical point of $\mathcal{C}$. This determines 
    the section of $(\rho_{\mathcal{D}})_{|_{\mathcal{D}}}$ unambiguously 
    on the vertex set of $\tilde{\mathcal{D}}$. The preimage of an edge $E$ of $\tilde{\mathcal{D}}$ 
    is a band (see Definition \ref{def:verticesEdgesPR}), which is a trivializable fibration with  
    compact segments as fibers over the interior of $E$. Therefore, one may extend 
    continuously the section from its boundary to the interior of $E$ in a canonical way 
    up to isotopies stabilizing each vertical line (see Figure \ref{fig1008}).
\end{proof}

\begin{figure}[H]
	\begin{center}
		\small
		\tikzstyle{every picture}=[scale=1.0*1]
		\input{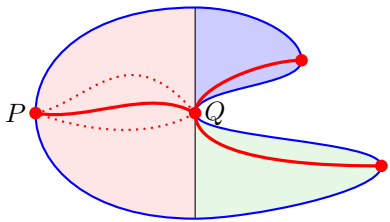}
	\end{center}		
	\caption{Decomposition in bands and choices of paths.}
	\label{fig1008}
\end{figure}

Note that without the genericity assumption,  the conclusion of Proposition \ref{prop:existsect} 
is not necessarily true, as may be checked on Figure \ref{fig:twoInteriorPts}.

\begin{definition}   
\label{def:PRgraphsource}
Let $\mathcal{D}$ be a domain of finite type in a vertical plane $(\mathcal{P}, \pi)$ 
with generic Poincar\'e--Reeb graph $\tilde{\mathcal{D}}$. Then any section  of 
$(\rho_{\mathcal{D}})_{|_{\mathcal{D}}} : \mathcal{D} \to \tilde{\mathcal{D}}$ is called 
a \defi{Poincar\'e--Reeb graph in the source} of $\mathcal{D}$. By contrast, the graph 
$\tilde{\mathcal{D}}$ is called the \defi{Poincar\'e--Reeb graph in the target}. 
\end{definition}

Using the notion of vertical equivalence defined in Subsection \ref{subs:equiv} below, 
one may show that any Poincar\'e--Reeb graph $\doubletilde{\mathcal{D}}$ 
in the source and the Poincar\'e--Reeb graph $\tilde{\mathcal{D}}$ in the target 
are vertically isomorphic: $\tilde{\mathcal{D}} \approx_v \doubletilde{\mathcal{D}}.$  
As explained above, an advantage of the latter construction is that the Poincar\'e--Reeb 
graph in the source lives inside the same plane as the generic finite type domain  $\mathcal{D}.$

Another advantage is that one may define Poincar\'e--Reeb graphs in the source even 
for algebraic domains which are not of finite type, but for which the affine projection $\pi$ 
is assumed to be generic in the sense of Definition \ref{def:genericDir}. 
In those cases the $\mathcal{D}$-collapse of the starting 
affine plane $\mathcal{P}$ is not any more homeomorphic to $\Rr^2$.

\subsection{Vertical equivalence}
\label{subs:equiv}        

The following definition of {\em vertical equivalence} 
is intended to capture the underlying combinatorial structure 
of subsets of vertical planes. That is, we consider that two vertically equivalent 
such subsets have the same combinatorial type.

\begin{definition}   
\label{def:CurveEquiv1}
Let $X$ and $X'$ be subsets of the vertical planes 
$(\mathcal{P},\pi)$ and $(\mathcal{P}',\pi')$ respectively. 
We say that $X$ and $X'$ are \defi{vertically equivalent}, 
denoted by $X\approx_v X'$, if there exist orientation preserving 
homeomorphisms $\Phi : \mathcal{P} \to \mathcal{P}'$ and $\psi : \Rr \to \Rr$ 
such that $\Phi(X) = X'$  and the following diagram is commutative: 
\begin{center}
   \begin{tikzcd}
         \mathcal{P} \arrow{d}[swap]{\pi} \arrow{r}{\Phi}    & \mathcal{P}' \arrow{d}{\pi'} \\
         \Rr \arrow{r}[swap]{\psi}   & \Rr
   \end{tikzcd}
\end{center}
\end{definition}

In the sequel we will apply the previous definition to situations when $X$ and $X'$ are either 
domains of finite type  in the sense of Definition \ref{def:domainFiniteType}  or transversal graphs 
in the sense of Definition \ref{def:transversalGraph}. 

\begin{proposition}   
\label{th:equiv}
Let $\mathcal{D}$ and $\mathcal{D}'$ be compact connected 
domains of finite type in vertical planes, with Poincar\'e--Reeb graphs $G$ and $G'$. 
Assume that both are generic in the sense of Definition \ref{def:gendom}. Then:
$$\mathcal{D} \approx_v \mathcal{D}' \iff G \approx_v G'.$$
\end{proposition}

Before giving the proof of Proposition \ref{th:equiv}, let us make some remarks:
\begin{itemize}
  \item Denote $\mathcal{C} = \partial \mathcal{D}$ and $\mathcal{C}' = \partial \mathcal{D}'$.
  We have $\Phi(\mathcal{C}) = \mathcal{C}'$.

  \item $\Phi$ sends the topological critical points $\{P_i\}$ of $\mathcal{C}$ 
  bijectively to the topological  critical points $\{P_i'\}$ of $\mathcal{C}'$.
  In fact, such a critical point may be geometrically characterized by the local behavior 
  of $\mathcal{D}$ relative to the vertical line through this point. A point $P$ is a 
  topological critical point of $\pi_{|\mathcal{C}}$, if and only if the intersection of $\mathcal{D}$ 
  with the vertical line $\ell$ through $P$ is a point in a neighborhood of $P$, 
  or a segment such that $P$ is in the interior of the segment.
  The homeomorphism $\Phi$ sends the vertical line $\ell$ to a vertical line $\ell'$ 
  and $\mathcal{D}$ to $\mathcal{D}'$, hence $P'=\Phi(P)$ is a topological critical point of 
  $\pi_{|\mathcal{C}'}$.

  \item The equivalence preserves the $\pi$-order of the critical points: 
  if $\mathcal{D} \approx_v \mathcal{D}'$, and if $P_i$, $P_j$ are critical points of 
  $\pi_{|\mathcal{C}}$ with $\pi(P_i) < \pi(P_j)$ then the corresponding critical points 
  of $\pi_{|\mathcal{C}'}$, $P'_i :=\Phi(P_i)$, $P'_j :=\Phi(P_j)$ verify $\pi'(P_i') < \pi'(P_j')$. 
  This comes from the assumption that the homeomorphisms $\Phi$ and $\psi$ 
  involved in Definition \ref{def:CurveEquiv1} are orientation preserving. 
\end{itemize}

\begin{example}
     Consider the canonical affine vertical plane $(\Rr^2, x)$ in the sense of 
     Definition \ref{def:vertplanes}. Then the vertical  equivalence 
     preserves the $x$-order, that is to say, if $x(P_i) < x(P_j)$ then $x(P_i') < x(P_j')$. 
     Notice that the $y$-order of the critical points may not be preserved. 
     However $\Phi$ preserves the orientation on each vertical line, 
     i.e.{} $y \mapsto \Phi(x_0,y)$ is a strictly increasing function.
\end{example}

\begin{example}   \label{ex:permutation}
     Consider again the canonical affine vertical plane $(\Rr^2, x)$ and 
     a generic algebraic domain $\mathcal{D}$ in it, homeomorphic to a disc. 
     Denote $\mathcal{C}=\partial\mathcal{D}$. It is homeomorphic to a circle. 
     Then the set of critical points of $\pi |_{\mathcal{C}}$ 
     (which are the same as the topological critical points, 
     by the genericity assumption) yields a permutation.
     To explain that, we will define two total orders on the set of critical points. 
     The first order enumerates $\{P_i\}$ in a circular manner following $\mathcal{C}$, 
     obtained by following the curve, starting with the point with the smallest $x$ coordinate,  
     the curve being oriented as the boundary of $\mathcal{D}$. 
     The second order is obtained by ordering the abscissas $x(P_i)$ using the standard 
     order relation on $\Rr$. Now, as explained by 
     Knuth (see \cite[page 17]{Gh1}, 
     \cite[Definition 4.21]{sorea2019permutations}, \cite[Section 1]{sorea2019constructing}), 
     two total order relations on a finite set give rise to a permutation  
     $\sigma$: in our case, $\sigma(i)$ is the rank of $x(P_i)$ 
     in the ordered list of all abscissa.
     The vertical equivalence preserves the permutation:
     if $\mathcal{D} \approx_v \mathcal{D}'$ then $\sigma=\sigma'$. However, the reverse implication 
     could be false, as shown in the picture below, which shows two generic 
     real algebraic domains homeomorphic to discs with the same permutation 
       $\left(\begin{smallmatrix}
           1 & 2 & 3 & 4 & 5 & 6 \\
           1 & 5 & 3 & 6 & 2 & 4
         \end{smallmatrix}\right)$, but which are not vertically equivalent, as may be seen by 
             considering their Poincar\'e-Reeb trees.

\begin{figure}[H]
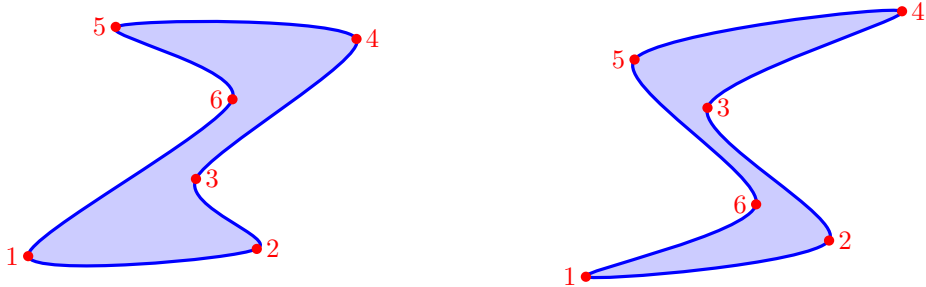

	\begin{minipage}{0.35\textwidth}
	\begin{center}
		\small
		\tikzstyle{every picture}=[scale=1.0*0.8]
		\input{newfigures/fig-reeb-210.tikz}
	\end{center}	
	\end{minipage}
	\qquad	\qquad
	\begin{minipage}{0.35\textwidth}
	\begin{center}
		\small
		\tikzstyle{every picture}=[scale=1.0*0.8]
		\input{newfigures/fig-reeb-211.tikz}
	\end{center}		
	\end{minipage}
	\caption{Two non-vertically equivalent real algebraic domains with the same permutation.} 
    \label{fig001}
 \end{figure}
\end{example}

Example \ref{ex:permutation} shows that the permutations are not complete invariants of 
generic domains of finite type homeomorphic to disks, under vertical equivalence. 
However, by Proposition \ref{th:equiv}, the Poincar\'e--Reeb graphs is a complete invariant for the vertical equivalence.

\medskip

\begin{proof}[Proof of Proposition \ref{th:equiv}]
~
\begin{itemize}
  \item  $\Rightarrow$. Suppose $\mathcal{D}\approx_v \mathcal{D}'$ and let 
     $\Phi : \mathcal{P}\to \mathcal{P}'$ be a homeomorphism realizing this equivalence 
     through a commutative diagram 
          \begin{center}
       \begin{tikzcd}
             \mathcal{P} \arrow{d}[swap]{\pi} \arrow{r}{\Phi}    & \mathcal{P}' \arrow{d}{\pi'} \\
             \Rr \arrow{r}[swap]{\psi}   & \Rr
       \end{tikzcd}
   \end{center}
     By definition, $\Phi$ preserves the vertical foliations, hence is compatible with the 
     vertical equivalence relations $\sim_{\mathcal{D}}$ and $\sim_{\mathcal{D}'}$ 
     of Definition \ref{def:equivRel}. Therefore it induces a homeomorphism 
     $\tilde{\Phi} : \tilde{\mathcal{P}} \to \tilde{\mathcal{P}}'$ from the 
     $\mathcal{D}$-collapse of $\mathcal{P}$ to the $\mathcal{D}'$-collapse of $\mathcal{P}'$, 
     sending 
     $G = \mathcal{D}/{\sim}$ to $G'=\mathcal{D}'/{\sim}$. This homeomorphism  
     gets naturally included in a commutative diagram
         \begin{center}
       \begin{tikzcd}
            \tilde{\mathcal{P}} \arrow{d}[swap]{\tilde{\pi}} \arrow{r}{\tilde{\Phi}}    
                    & \tilde{\mathcal{P}}' \arrow{d}{\tilde{\pi}'} \\
             \Rr \arrow{r}[swap]{\psi}   & \Rr
       \end{tikzcd}
   \end{center}
 Therefore, by Definition \ref{def:CurveEquiv1},   $G \approx_v G'$.

  \item $\Leftarrow$. The keypoint is to reconstruct the topology of a generic 
      domain of finite type $\mathcal{D}$ homeomorphic to a disk (and of its boundary $\mathcal{C}$) 
      from its Poincar\'e--Reeb graph $G$. To this end, one may construct a kind of tubular 
      neighborhood $\overline{\mathcal{D}}$ of $G$, obtained by thickening it using 
      vertical segments (see Figure \ref{fig1009}). 
     Then $\overline{\mathcal{D}}$  is vertically equivalent to $\mathcal{D}$.

\begin{figure}[H]
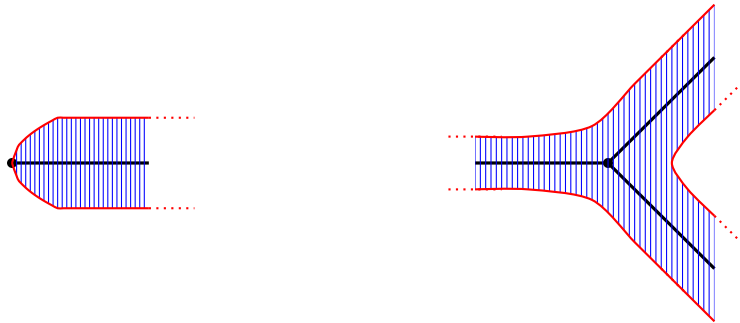

	\begin{minipage}{0.35\textwidth}
	\begin{center}
		\small
		\tikzstyle{every picture}=[scale=1.0*0.6]
		\input{newfigures/fig-reeb-212.tikz}
	\end{center}	
	\end{minipage}
	\qquad	
	\begin{minipage}{0.35\textwidth}
	\begin{center}
		\small
		\tikzstyle{every picture}=[scale=1.0*0.7]
		\input{newfigures/fig-reeb-213.tikz}
	\end{center}		
	\end{minipage}	

	\caption{Thickening in the neighborhood of 
	      an exterior vertex (left) and of an interior vertex (right).}
	\label{fig1009}
\end{figure}

Now suppose that $G \approx_v G'$ and let 
$\tilde{\Phi} : \tilde{\mathcal{P}} \to \tilde{\mathcal{P}}'$ 
be a homeomorphism inducing this equivalence. This homeomorphism 
induces also a vertical equivalence of convenient such thickenings, 
hence yields the equivalence $\mathcal{D} \approx_v \mathcal{D}'$.
\end{itemize}
\end{proof}

The combinatorial types of \emph{generic} transversal graphs can be realized by special types of graphs with smooth edges in the canonical affine vertical plane $(\Rr^2, x:\Rr^2\rightarrow\Rr)$:

\begin{proposition} 
\label{prop:smoothrealgtg}
Any generic transversal graph in a vertical plane is vertically equivalent to a graph in the canonical affine vertical plane, whose edges are moreover smooth and 
smoothly transversal to the vertical lines.
\end{proposition}

We leave the proof of this proposition to the reader.

\begin{remark}
We said at the beginning of this subsection that we introduced vertical equivalence as a way to capture the combinatorial aspects of subsets of vertical planes. It is easy to construct a combinatorial object (that is, a structure on a finite set) which encodes the combinatorial type of a generic transversal graph. For instance, given such a graph $G$, one may number its vertices from $1$ to $n$ in the order of the values of the vertical projection $\pi$. Then, for each edge $\alpha$ of $G$, one may remember both its end points $a < b$ and, for each number  $c \in \{a+1, \dots, b-1\}$, whether $\alpha$ passes below or above the vertex numbered $c$.    
\end{remark}

\section{Algebraic realization in the compact connected case}
\label{sec:algebraic}

In this section we give the main result of the paper, Theorem \ref{th:realizationA}: given a compact connected generic transversal graph $G$ in a vertical plane 
(see Definition \ref{def:transversalGraph}), we prove that there exists a compact 
algebraic domain in the canonical affine vertical plane whose Poincar\'e--Reeb graph is 
vertically equivalent to $G$.
We will prove a variant of Theorem \ref{th:realizationA} for non-compact graphs 
in the next section (Theorem  \ref{th:realizationB}).

Using the canonical orientation of the target $\mathbb{R}$ of the vertical projection, one may 
distinguish two kinds of interior and exterior vertices of the graph $G$ (see Figure \ref{fig1011}).

\begin{figure}[H]
	\begin{center}
		\small
		\tikzstyle{every picture}=[scale=1.0*1]
		\input{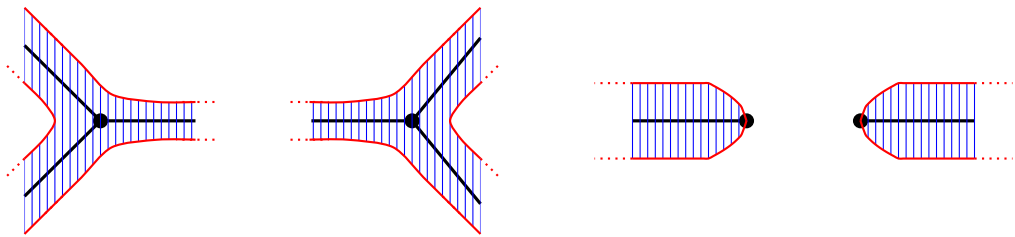}
	\end{center}	

	\caption{The two kinds of interior vertices (on the left) and of exterior vertices (on the right).}
	\label{fig1011}
\end{figure}

Our strategy of proof of Theorem \ref{th:realizationA}  is as follows:
\begin{itemize}
  \item we realize the generic transversal graph $G$ as a Poincar\'e--Reeb graph of a 
      finite type domain defined by a smooth function;
  \item we present a Weierstrass-type theorem that approximates 
      any smooth function by a polynomial function;
  \item we adapt this Weierstrass-type theorem in order to control vertical tangents, 
       and we realize $G$ as the Poincar\'e--Reeb graph of a generic finite type algebraic domain.
\end{itemize}

\subsection{Smooth realization}

First, we construct a smooth function $f$ that realizes a given generic transversal graph.

\begin{proposition}
\label{prop:smooth-real}
   Let $G$ be a compact connected generic transversal graph.
   There exists a  $C^{\infty}$  
   function $f : \Rr^2 \to \Rr$ such that the curve $\mathcal{C} = (f=0)$ 
   does not contain critical points of $f$ and is the boundary 
   of a domain of finite type whose Poincar\'e--Reeb graph in 
   the canonical vertical plane $(\Rr^2, x)$ is vertically equivalent to $G$. 
\end{proposition}

\begin{figure}[H]
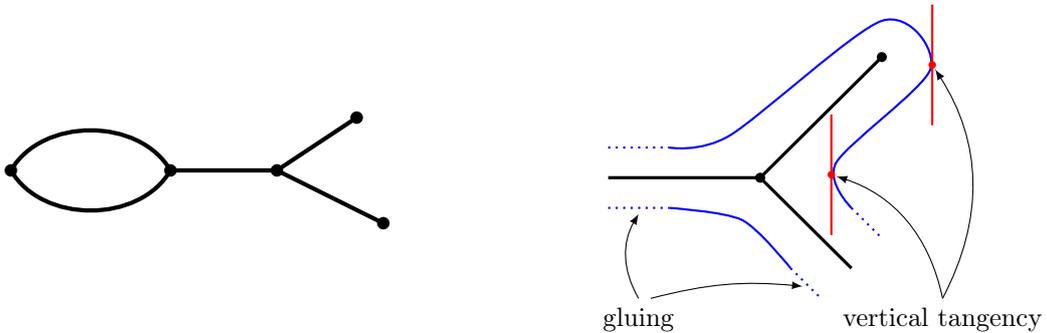

	\begin{minipage}{0.35\textwidth}
	\begin{center}
		\small
		\tikzstyle{every picture}=[scale=1.0*0.7]
		\input{newfigures/fig-reeb-302.tikz}
	\end{center}	
	\end{minipage}
	\qquad	\qquad \qquad
	\begin{minipage}{0.35\textwidth}
	\begin{center}
		\small
		\tikzstyle{every picture}=[scale=1.0*0.8]
		\input{newfigures/fig-reeb-303.tikz}
	\end{center}		
	\end{minipage}	

	\caption{A generic compact transversal graph (left) and a local smooth realization (right).}
	\label{fig302-303}
\end{figure}

\begin{proof} 
	
The idea is to construct first the curve $\mathcal{C}$, then the function $f$.
We construct $\mathcal{C}$ by interpolating between local constructions in the neighborhoods of 
the vertices of $G$ (see Figure \ref{fig302-303}). Let us be more explicit.
We may assume, up to performing a vertical equivalence, that $G$ is a 
	graph with smooth compact edges in the canonical affine vertical plane $(\mathbb{R}^2, x)$, 
	whose edges are moreover smoothly transversal to the verticals (see 
	Proposition \ref{prop:smoothrealgtg}).  
	Let $\epsilon > 0$ be fixed.
Then, one may construct $\mathcal{C}$ verifying the following properties:
\begin{itemize}
	\item $\mathcal{C}$ is compact;
	\item $\mathcal{C} \subset N(G,\epsilon)$: the curve is contained in  
	the $\epsilon$-neighborhood of $G$;
	\item $\mathcal{C}$ has only one vertical tangent associated to each vertex of $G$;
	\item all these tangents are ordered in the same way as the vertices of $G$.
\end{itemize}
Note that this last condition is automatic once $\epsilon$ is chosen less than half 
	the minimal absolute value $| x(P_i) - x(P_j)|$, where $P_i$ and $P_j$ are distinct 
	vertices of $G$.
	
Once $\mathcal{C}$ is fixed,  one may construct $f$ by following the steps:
\begin{itemize}
	\item  Bicolor the complement $\mathbb{R}^2   \setminus  \mathcal{C}$ 
	of $\mathcal{C}$ using the numbers $\pm 1$, such that 
	neighboring connected components have distinct associated numbers. Denote 
	by $\sigma : \mathbb{R}^2  \setminus \mathcal{C} \to \Rr$ the resulting function. 
	\item Choose pairwise distinct annular neighborhoods $N_i$ of the connected 
	components $\mathcal{C}_i$ of $\mathcal{C}$, and diffeomorphisms 
	$ \phi_i : N_i \simeq \mathcal{C}_i \times (-1, 1)$ such that $p_2 \circ \phi_i$			
	(the composition 	of the second projection $p_2 :  \mathcal{C}_i \times (-1, 1) \to (-1, 1)$ 
	and of $\phi_i$)
	has on the complement of $\mathcal{C}_i$ the same sign as $\sigma$.            
	\item For each connected component $S_j$ of $\mathbb{R}^2  \setminus  \mathcal{C}$, 
	consider the open set $U_j \subset S_j$ obtained as the complement of the union 
	of annuli of the form $\phi_i^{-1}(\mathcal{C}_i \times [-1/2, 1/2])$. Then consider the 
	restriction  $\sigma_j : U_j \to \mathbb{R}$ of $\sigma$ to $U_j$.
	\item Fix a smooth partition of unity on $\mathbb{R}$ subordinate to the locally finite open 
	covering consisting of the annuli $N_i$ and the sets $U_j$. Then glue the smooth functions 
	$p_2 \circ \phi_i : N_i \to \mathbb{R}$ and $\sigma_j : U_j \to \mathbb{R}$ using it.
	\item The resulting function $f$ satisfies the desired properties.
\end{itemize}
\end{proof}

\subsection{A Weierstrass-type approximation theorem}

Let us first recall the following classical result:

\begin{theorem}[Stone-Weierstrass theorem, \cite{Stone}]
\label{th:weierstrass}
    Let $X$ be a compact Hausdorff space. Let $C(X)$ be the Banach $\Rr$-algebra of continuous 
    functions from $X$ to $\Rr$ endowed with the norm $\|\cdot\|_\infty$.
    Let $A \subset C(X)$ be such that:
   \begin{itemize}
       \item $A$ is a sub-algebra of $C(X)$,
       \item $A$ separates points (that is, for each $x,y \in X$ with $x\neq y$ 
            there exists $f \in A$ such that $f(x) \neq f(y)$), 
       \item for each $x \in X$, there exists $f \in A$ such that $f(x) \neq 0$.
   \end{itemize}
     Then $A$ is dense in $C(X)$ relative to the norm $\|\cdot\|_\infty$.
\end{theorem}

We will only use the previous theorem through the following corollary:

\begin{corollary}   
\label{cor:approxpol}
Let $f : \Rr^2 \to \Rr$ be a continuous map and $a,b\in\Rr, a<b$.
For each $\epsilon>0$, there exists a polynomial $p \in \Rr[x,y]$ such that :
$$\forall (x,y) \in [a,b]\times [a,b] \quad \big| f(x,y) - p(x,y) \big| < \epsilon$$
\end{corollary}

\begin{proof}
We apply Theorem \ref{th:weierstrass} with $X = [a,b]\times [a,b]$,  $A=\Rr[x,y]$. 
This set $A$ satisfies the three conditions of Theorem \ref{th:weierstrass} (the last one because $1_X \in A$), 
therefore $A$ is dense in $C(X)$, which implies that $f$ can indeed be uniformly 
arbitrarily well approximated on $X$ by polynomials.
\end{proof}

Is Corollary \ref{cor:approxpol} sufficient to answer the realization question? No!
Indeed, even if it provides a polynomial $p(x,y)$ such that $(p(x,y)=0)$ {lies in a close 
neighborhood of $(f(x,y)=0)$, we have no control on the vertical tangents of the algebraic 
curve $(p=0)$, whose Poincar\'e--Reeb graph can therefore be more complicated than 
the Poincar\'e--Reeb graph of $(f=0)$.
In the sequel we construct a polynomial $p$ by keeping at the same time a control on 
the vertical tangents of a suitable level curve of it.

\subsection{Algebraic realization}     

\begin{proposition}
\label{prop:alg-real}
Let $f : \Rr^2 \to \Rr$ be a $C^3$ function such that $\mathcal{C} = (f=0)$ 
is a curve which does not contain critical points of $f$, which has only simple vertical tangents, 
and is included in the interior of a compact subset $K$ of $\mathbb{R}^2$.
For each $\delta>0$, there exists a polynomial $p\in \Rr[x,y]$ such that (see Figure \ref{fig03}):
\begin{itemize}
    \item $(p=0) \cap K \subset N(f=0,\delta)$,
    \item for each point $P_0 \in (f=0)$ where $(f=0)$ has a vertical tangent, there exists 
         a unique $Q_0 \in (p=0)$ in the disc $N(P_0,\delta)$ centered at $P_0$ and of radius 
           $\delta$ such that $(p=0)$ has also a vertical tangency at $Q_0$,
    \item $(p=0) \cap K$ has no vertical tangent except at the former points.
\end{itemize} 
\end{proposition}

\begin{figure}[H]
	\begin{center}
		\small
		\tikzstyle{every picture}=[scale=1.0*1.5]
		\input{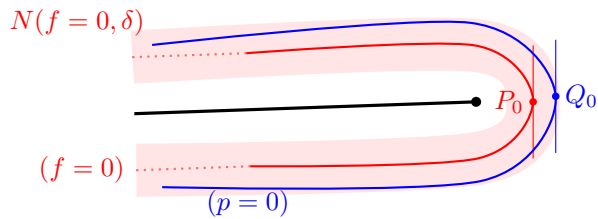}
	\end{center}		
	\caption{Algebraic realization.}
	\label{fig03}
\end{figure}
We prove this proposition in Subsection \ref{prop:proofpropctrltang} below.

By taking the numbers $\epsilon > 0$ and $\delta > 0$ appearing in Propositions 
\ref{prop:smooth-real} and \ref{prop:alg-real} sufficiently small, we get:

\begin{theorem}
\label{th:realizationA}
    Any compact connected generic transversal graph can be realized as the Poincar\'e--Reeb 
   graph of a connected algebraic domain of finite type.
\end{theorem}

\begin{proof}[Proof of the theorem]
    Starting with a compact connected transversal generic graph $G$, 
    it can be realized by a smooth function $f$ (Proposition \ref{prop:smooth-real}), 
    which in turn can be replaced by a polynomial map $p$ (Proposition \ref{prop:alg-real}). 
\end{proof}

The referees ask:
\begin{question*}
Is it possible to estimate the degree of a polynomial defining the algebraic domain in term of the combinatorics of the graph $G$?
\end{question*}
A referee suggested to use a degree effective version of the $\mathcal{C}^k$ Weierstrass polynomial approximation theorem, as in \cite[Theorem 2]{bagby}. 
Furthermore, in \cite{lerarioStecconi}, the authors construct an algebraic hypersurface that approximates a smooth compact hypersurface with a control of its minimal degree in terms of geometric data of the hypersurface.

\subsection{Proof of Proposition \ref{prop:alg-real}}  
\label{prop:proofpropctrltang}         

\textbf{Compact support.}
Let $M>0$ such that $(f=0) \subset [-(M-1),M-1]^2$ (remember that $(f=0)$ is 
assumed to be included in a compact set).
We replace the function $f$ by a function $g$ with compact support. 
More precisely, let $g : \Rr^2 \to \Rr$ be a function such that:
\begin{itemize}
    \item $g$ is $C^3$,
    \item $f=g$ on $[-(M-1),M-1]^2$,
    \item $g=0$ outside $(-M,M)^2$,
    \item $g$ does not vanish in the intermediate zone (hatched area below).
\end{itemize}

\begin{figure}[H]
	\begin{center}
		\small
		\tikzstyle{every picture}=[scale=1.0*0.5]
		\input{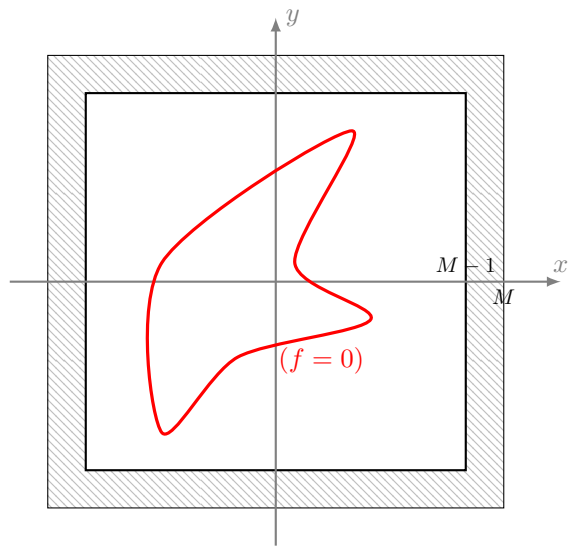}
	\end{center}		
	\caption{Compact support of $g$.}
	\label{fig04}
\end{figure}

Such a function may be constructed using an adequate partition of unity.

\textbf{Polynomial approximation of $g$ and $f$.}
We need a polynomial $p$ approximating $g$, but also that some partial derivatives of $p$ approximate the corresponding partial derivatives of $g$.
This can be done by a $\mathcal{C}^k$ Weierstrass polynomial approximation. More precisely, one can use the density of polynomials in the $\mathcal{C}^k$ topology, as stated in \cite[Proposition 1.3.7.]{DeLellis2019}. Nevertheless, we state such a result and emphasise which partial derivatives we need to approximate. 

\begin{lemma}
\label{lem:approx}
Let us fix $\epsilon>0$. There exists a polynomial $p \in \Rr[x,y]$ such that,
for all $(x,y) \in [-M,M]^2$:
$$
\left| p(x,y) - g(x,y) \right| \le (2M)^3\epsilon, \qquad
\left| \partial_y p(x,y) - \partial_y g(x,y) \right| \le (2M)^2\epsilon,$$
$$
\left| \partial_x p(x,y) - \partial_x g(x,y) \right| \le (2M)^2\epsilon,
\qquad \left| \partial_{y^2} p(x,y) - \partial_{y^2} g(x,y) \right| \le 2M\epsilon.
$$
\end{lemma}

In order to be self-contained we give a short proof, inspired by \cite{Elredge}.
\begin{proof}
By Corollary \ref{cor:approxpol} 
applied to the function $\partial_x\partial_y\partial_y g$ and to $(a,b) = (-M, M)$, 
there exists a polynomial $p_0 \in \Rr[x,y]$ such that:
$\forall (x,y) \in [-M,M]^2$ $\left| p_0(x,y) - \partial_x\partial_y\partial_y g(x,y) \right| < \epsilon.$
Now our polynomial $p \in \Rr[x,y]$ is defined by a triple integration :
    $$p(x,y) = \int_{-M}^{x} \int_{-M}^{y} \int_{-M}^{y} p_0(u,v) \, dv \,dv\, du.$$
    
We start by proving the last inequality.
By Fubini theorem:
$\partial_{y^2} p(x,y) = \int_{-M}^{x} p_0(u,y) \, du.$
Therefore:
$$\left| \partial_{y^2} p(x,y) - \partial_{y^2} g(x,y) \right|
= \left| \int_{-M}^{x} \big(p_0(u,y) - \partial_x\partial_{y^2} g(u,y)\big)\, du  \right|
\le \left| \int_{-M}^{x} \epsilon \,du \right| \le 2M\epsilon.$$
The first equality is a consequence of the fact that:
$\int_{-M}^{x}\partial_x\partial_{y^2} g(u,y)\, du
= \partial_{y^2} g(x,y) - c(y)$
where $c(y) = \partial_{y^2} g(-M,y)$. 
As $g$ vanishes outside $(-M,M)^2$, for those points we have
$\partial_{y^2} g(x,y)=0$ so that $c(y)=0$.  
The inequality following it results from the definition of the polynomial $p_0$.
By successive integrations we prove the other inequalities. 
\end{proof}

Inside the square $[-M,M]^2$ the curve $(p=0)$ defined for a sufficiently small $\epsilon$ is in a neighborhood of $(f=0)$. However, remark that $(p=0)$ can also vanish outside the square $[-M,M]^2$.

\bigskip

\textbf{The curve $(p=0)$ inside the square.}

Let us explain the structure of the curve $(p=0)$ around a point $P_0 \in (f=0)$ where the tangent is not vertical (recall that $f=g$ inside the square $[-(M-1),M-1]^2$). 
\begin{itemize}
	\item Fix $\delta>0$. Let $B(P_0,\delta)$ be a neighborhood of $P_0$.
	On this neighborhood $f$ takes positive and negative values.
	\item Let $\eta>0$ and $Q_1,Q_2 \in B(P_0,\delta)$ such that 
	$f(Q_1)>\eta$ and $f(Q_2)<-\eta$.
	\item We choose the $\epsilon$ of Lemma \ref{lem:approx} such that $(2M)^3 \epsilon<\eta/2$. 
	\item $p(Q_1)>f(Q_1)-(2M)^3\epsilon>\eta/2>0$; a
	 similar computation gives $p(Q_2)<0$, hence $p$ vanishes 
	 at a point $Q_0 \in [Q_1Q_2]\subset B(P_0,\delta)$.
	\item Because we supposed $\partial_y f \neq 0$ in 	$B(P_0,\delta)$,
	we also have $\partial_y p \neq 0$. Hence $(p=0)$ is a smooth simple curve 
	in $B(P_0,\delta)$ with no vertical tangent.
\end{itemize}

\begin{figure}[H]
	\begin{center}
		\small
		\tikzstyle{every picture}=[scale=1.0*0.7]
		\input{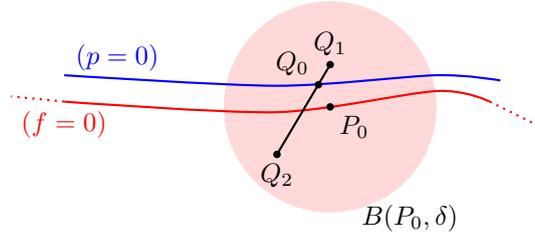}
	\end{center}		
	\caption{Existence of the curve $(p=0)$.}
	\label{fig05}
\end{figure}

Notice that the construction of $(p=0)$ in $B(P_0,\delta)$ depends on $\epsilon$, whose choice depends on the point $P_0$. To get a common choice of $\epsilon$, we first cover the compact curve $(f=0)$ by a finite number of balls $B(P_0,\delta)$ and take the minimum of the $\epsilon$ above.

\bigskip
 
\textbf{Vertical tangency.}

\begin{itemize}
    \item Let $P_0 = (x_0,y_0)$ be a point with a simple vertical tangent of $(f=0)$, that is to say:
    $$\partial_y f(x_0,y_0) = 0 \qquad 
    \partial_x f(x_0,y_0) \neq 0 \qquad
    \partial_{y^2} f(x_0,y_0) \neq 0$$

	\item For similar reasons as before, $(p=0)$ is a non-empty smooth curve passing near $P_0$.
	
    \item In the following we may suppose that the curve $(f=0$) is locally at the left of its vertical tangent, that is to say:
    $$\partial_x f(x_0,y_0) \times \partial_{y^2} f(x_0,y_0) > 0$$

	An example of this behavior is given by $f(x,y) := x+y^2$ at $(0,0)$.
	
	\begin{figure}[H]
	\begin{center}
		\small
		\tikzstyle{every picture}=[scale=1.0*0.7]
		\input{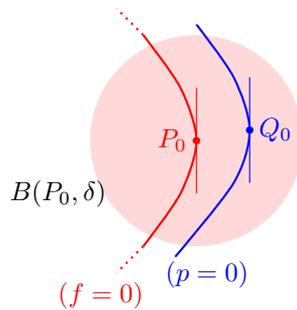}
	\end{center}		
	\caption{Vertical tangent.}
	\label{fig06}
	\end{figure}  

	\item Fix $\delta>0$. Let $B(P_0,\delta)$ be a neighborhood of $P_0$.

    \item $\partial_y p \sim \partial_y f$. As $\partial_y f$ vanishes at the point $P_0$ of $(f=0)$, then $\partial_y f$ takes positive and negative values near this point. Let $\eta>0$, and $Q_1 \in (f=0)$ such that 
    $\partial_y f(Q_1)>\eta$. For a sufficiently small $\epsilon$, there exists $R_1 \in (p=0)$ such that $\partial_y f(R_1)>\frac23\eta$. Therefore $\partial_y p(R_1)>\frac13\eta>0$. For a similar reason there exists $R_2\in (p=0)$ such that $\partial_y p(R_2)<0$. Then there exists $Q_0 \in (p=0) \cap B(P_0,\delta)$ such that $\partial_y p(Q_0)=0$.
    
    \item $\partial_x p \sim \partial_x f$. As $\partial_x f(P_0) \neq 0$, one has also  
       $\partial_x p(Q_0) \neq 0$, thus $p$ has a vertical tangent at $Q_0$.
    
    \item $\partial_{y^2} p \sim \partial_{y^2} f$ and they do not vanish near $P_0$ and $Q_0$, 
       therefore the vertical tangent at $Q_0$ for $(p=0)$ is simple and has the same type 
       as the vertical tangent at $P_0$ for $(f=0)$.
    
    \item Moreover as $\partial_{y^2} p \neq 0$ on $(p=0) \cap B(P_0,\delta)$, 
       thus $\partial_y p$ vanishes only once, hence there is only one vertical tangent 
       in this neighborhood.

\end{itemize}

\section{Algebraic realization in the non-compact and connected case}
\label{sec:noncompact}

\subsection{Domains of weakly finite type in vertical planes}    

We consider an algebraic domain $\mathcal{D} \subseteq \Rr^2$ in the sense of Definition 
\ref{def:algebraicDomain},  whose boundary $\mathcal{C} := \partial \mathcal{D}$ is a  
connected but non-compact curve. This curve $\mathcal{C}$ is homeomorphic to a line 
and has two branches at infinity (the germs at infinity of the two connected components of 
$\mathcal{C}\setminus K$, where $K \subset \mathcal{C}$ is a non-empty compact arc). 
Let us suppose that these branches are in generic position w.r.t.~the vertical direction: 
none of them has a vertical asymptote.
This leads us to  Definition \ref{def:domainOfWeaklyFiniteType} below, which represents 
a generalization of the notion of domain of finite type (see Definition \ref{def:domainFiniteType}), 
since we only ask $\pi_{|_{\mathcal{C}}}: \mathcal{C}\to \Rr$ to be proper, 
allowing $\pi_{|_{\mathcal{D}}}: \mathcal{D}\to \Rr$ not to be so. In turn, the genericity notion 
is an analog of that introduced in Definition \ref{def:gendom}.

\begin{definition}\label{def:domainOfWeaklyFiniteType}
Let $(\mathcal{P},\pi)$ be a vertical plane. Let $\mathcal{D}\subset\mathcal{P}$ be a closed subset 
homeomorphic to a surface with non-empty boundary. Denote by $\mathcal{C}$ its boundary. 
We say that $\mathcal{D}$ is a 
 \defi{domain of weakly finite type} in $(\mathcal{P},\pi)$ if:
\begin{enumerate}
    \item \label{weaklyproperness} 
         the restriction $\pi_{|_{\mathcal{C}}}: \mathcal{C}\to \Rr$  is proper;
    \item  \label{weaklyfincrit}
        the topological critical set $\Sigma_{\text{top}} (\mathcal{C})$ is finite.
\end{enumerate}
   Such a domain is called \defi{generic} if no two topological critical points of $\mathcal{C}$ 
   lie on the same vertical line. A \emph{Poincar\'e--Reeb graph} of a generic 
   domain of weakly finite type is one of its Poincar\'e--Reeb graphs in the source in the sense 
   of Subsection \ref{ssec:PRsource}.
\end{definition}

For instance, the closed upper half-plane $\mathbb{H}$ in $(\Rr^2, x)$ is a generic domain of weakly finite type (for which $\Sigma_{\text{top}} (\mathcal{C}) = \emptyset$). 
Its Poincar\'e-Reeb graphs are the sections of the restriction $x : \mathbb{H} \to \mathbb{R}$ 
of the vertical projection.

\subsection{The combinatorics of non-compact Poincar\'e--Reeb graphs} 
\label{ssec:combnon-compPR}    
   
Let $\mathcal{D}$ be a domain of weakly finite type in a vertical plane $(\mathcal{P}, \pi)$. 
When $\mathcal{C}$ is homeomorphic to a line, we distinguish three cases, 
depending on the position of $\mathcal{D}$ and of the branches of $\mathcal{C}$. We enrich  
the Poincar\'e--Reeb graph, by adding \emph{arrowhead vertices} representing  
directions of escape towards infinity. Moreover, the unbounded edges 
are decorated with \emph{feathers} oriented upward or downward, to indicate 
the unbounded vertical intervals contained in the domain.

\textbf{Case A.} \emph{One arrow.}

\begin{figure}[H]
	\begin{center}
		\small
		\tikzstyle{every picture}=[scale=1.0*0.9]
		\input{newfigures/fig-reeb-401.tikz}
	\end{center}	
  \caption{Case A.}
\end{figure}

In case A, the two branches of $\mathcal{C}$ are going in the same direction 
(to the right or to the left, as defined by the orientations of $\mathcal{P}$ and the target 
line $\Rr$ of $\pi$), $\mathcal{D}$ being in between.  Then we get a Poincar\'e--Reeb graph 
with one arrow (and no feathers).

\textbf{Case B.} \emph{Two arrows.}

\begin{figure}[H]
	\begin{center}
		\small
		\tikzstyle{every picture}=[scale=1.0*0.9]
		\input{newfigures/fig-reeb-402.tikz}
	\end{center}	
  \caption{Case B.}\label{fig:caseB}
\end{figure}

In case B, the two branches have opposite directions. Then we have a Poincar\'e--Reeb 
graph with two arrows, each arrow-headed edge being decorated with feathers 
(above or below), indicating the non-compact vertical intervals of type 
$[0,+\infty[$ or $]-\infty,0]$ contained in the domain bounded by that edge.

\textbf{Case C.} \emph{Three arrows.}

\begin{figure}[H]
	\begin{center}
		\small
		\tikzstyle{every picture}=[scale=1.0*0.8]
		\input{newfigures/fig-reeb-403.tikz}
	\end{center}	
  \caption{Case C.}
\end{figure}

In case C, where the two branches are going to the same direction but $\mathcal{D}$ 
is in the ``exterior'', we have a graph with three arrows: two arrows with simple feathers 
(for the vertical intervals of type $[0,+\infty[$ or $]-\infty,0]$) and one arrow with double 
feathers (for the vertical intervals of type $]-\infty,+\infty[$).

\begin{remark}
~ 
\begin{itemize}
  \item We can avoid the contraction of non-compact vertical intervals in the construction of the Poincar\'e--Reeb graph in case B and case C, in order to still have a graph $G$ naturally embedded in an affine plane.
  We first define a subset $\mathcal{H} \subset \Rr^2$ that contains $\mathcal{C}$, whose boundary $\partial \mathcal{H} = H_+ \cup H_-$ is the union of two curves homeomorphic to $\Rr$, transverse to the vertical foliation (one above $\mathcal{C}$, one below $\mathcal{C}$).
  
\begin{figure}[H]
	\begin{center}
		\small
		\tikzstyle{every picture}=[scale=1.0*0.6]
		\input{newfigures/fig-reeb-404.tikz}
	\end{center}	
\end{figure}  
  
  We change Definition \ref{def:equivRel}, by contracting vertical intervals of 
  $\mathcal{D} \cap \mathcal{H}$ (instead of vertical intervals of 
  $\mathcal{D}$) :  $P\sim_{\mathcal{D}} Q$ if $\pi(P)=\pi(Q):=x_0$ and 
  $P$ and $Q$ are on the same connected component of 
  $\mathcal{D} \cap \mathcal{H} \cap \pi^{-1}(x_0)$.  
    
  \item The feather decoration on non-arrowheaded edges can be recovered from feathers 
      at the other arrows and are omitted. 

  \item The cases A and C are complementary (or dual of each other). We can pass from 
      one to the other by considering $\mathcal C$ as the boundary of $\mathcal D$ or 
     of $\Rr^2 \setminus \mathcal D$.
 
  \item From this point of view, case B is its own complementary case. More on such complementarities will be said later (see Section \ref{sec:general}).
\end{itemize}
\end{remark}

\begin{proposition}\label{prop:treeNonCompact}
      Let $\mathcal{D}$ be a simply connected generic domain of weakly finite 
      type in a vertical plane. 
       Then its Poincar\'e--Reeb graph is a generic transversal binary tree.
\end{proposition}

\begin{proof}
     After applying a vertical equivalence in the sense of Definition \ref{def:CurveEquiv1}, 
     we may assume that $\mathcal{D}$ is embedded in the canonical vertical plane 
     $(\mathbb{R}^2, x)$. 

   Denote $\mathcal{C}:=\partial\mathcal{D}.$ The idea is to intersect $\mathcal{C}$ 
    (and $\mathcal{D}$) with a sufficiently big compact convex topological 
    disk $K$, to apply our previous construction for 
    $\mathcal{D} \cap K$, then to add suitable arrows.  In the figure below, such a disk is 
    represented as a Euclidean one, but one has to keep in mind that its shape may be different, 
    for instance a rectangle, in order to achieve topological transversality between its boundary and 
    the curve $\mathcal{C}$. 
    
    \begin{figure}[H]
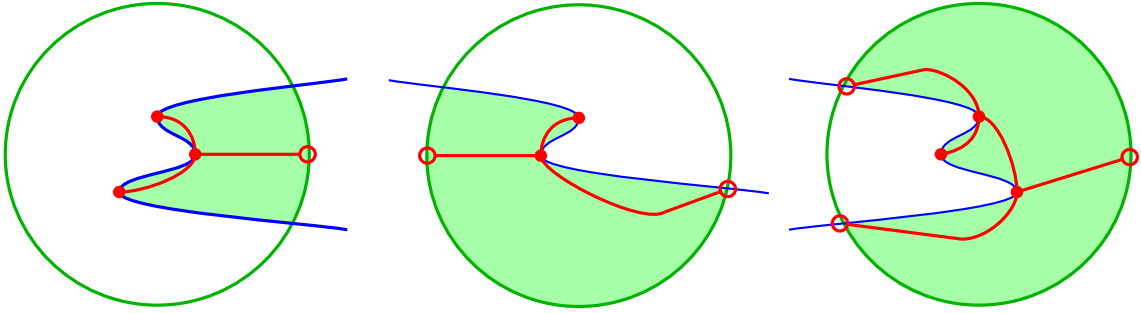

  \begin{minipage}{0.3\textwidth}
	\begin{center}
		\small
		\tikzstyle{every picture}=[scale=1.0*1]
		\input{newfigures/fig-reeb-405.tikz}
	\end{center}
  \end{minipage}\quad
  \begin{minipage}{0.3\textwidth}
	\begin{center}
		\small
		\tikzstyle{every picture}=[scale=1.0*1]
		\input{newfigures/fig-reeb-406.tikz}
	\end{center}
  \end{minipage}\quad
  \begin{minipage}{0.3\textwidth}
	\begin{center}
		\small
		\tikzstyle{every picture}=[scale=1.0*1]
		\input{newfigures/fig-reeb-407.tikz}
	\end{center}
  \end{minipage}
  \caption{Cases A, B, C (from left to right). The filled region is the compact 
     domain of finite type 
    $\mathcal{D}' := \mathcal{D} \cap K$. A Poincar\'e--Reeb graph in the source is also displayed. 
    The Poincar\'e--Reeb graph of $\mathcal{D}$ is obtained by replacing each circled 
    vertex by an arrow.}
\end{figure}

    First, notice that the case where $\mathcal{C}$ is compact is already known 
    (see Propositions \ref{prop:whenisPRgeneric} and \ref{prop:treeCompact}). 
    Assume therefore that $\mathcal{C}$ is a non-compact curve. 
    Then $\pi_{|\mathcal{C}}$ has a finite number of topological critical points. 
    We consider a sufficiently large convex compact topological disk $K$ 
    that contains all these critical points. Let $\mathcal{D}' := \mathcal{D} \cap K$ and 
    $\mathcal{C}' := \partial \mathcal{D}'$. We are then in the compact situation studied 
    before. By Proposition \ref{prop:treeCompact},  
    the Poincar\'e--Reeb graph of $\mathcal{D}'$ is a tree. 
    We add arrows (at each circled dot below) depending on each case.
\end{proof}
 
We extend now the notion of vertical equivalence of transversal graphs from 
Definition \ref{def:CurveEquiv1} to enriched non-compact transversal graphs, requiring that arrowhead vertices are sent to arrowhead vertices. Then we have the following generalization of Theorem 
\ref{th:equiv}, whose proof is similar:
   
\begin{proposition}
\label{prop:equivGeneral}
      Let $\mathcal{D}$, $\mathcal{D}'$ be generic simply connected domains of weakly 
      finite type. Then:
      $$\mathcal{D} \approx_v \mathcal{D}' \iff G \approx_v G'.$$
\end{proposition}

\subsection{Algebraic realization}     

We extend our realization thorem (Theorem \ref{th:realizationA})
of generic transversal graphs as Poincar\'e--Reeb graphs of algebraic domains 
 to the simply connected but non-compact case. 
The idea is to use the realization from the compact setting and consider the union with a line 
(or a parabola); finally, we take a neighboring curve. 

\begin{example}\label{ex:addBranches}
Here is an example, see Figure \ref{fig:exAddBranches}: 
starting from an ellipse $(f=0)$, we consider the union with a line $(g=0)$, 
then the unbounded component of $(f g = \epsilon)$ is a non-compact curve with two branches 
that have the shape of the ellipse on a large arc, if the sign of $\epsilon$ is conveniently 
chosen.

\begin{figure}[H]
	\begin{center}
		\small
		\tikzstyle{every picture}=[scale=1.0*1]
		\input{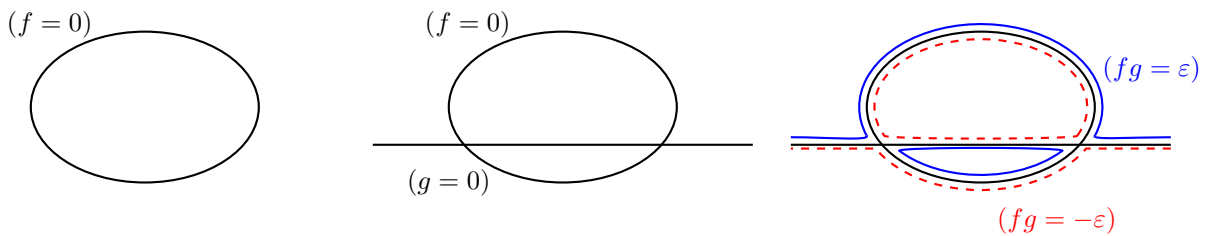}
	\end{center}
  \caption{Adding two branches to an ellipse.\label{fig:exAddBranches}}
\end{figure}

\end{example}

\begin{theorem}
\label{th:realizationB}
  Let $G$ be a connected, non-compact, 
  generic transversal tree in a vertical plane, 
  with at most three unbounded edges, not all on the same side (left or right), 
  enriched with compatible arrows and feathers (like in cases A, B or C 
  of Section \ref{ssec:combnon-compPR}). Let $G'$ be the compact tree obtained from $G$, by replacing each arrow by a sufficiently long edge with a circle vertex at the extremity. If $G'$ can be realized by a connected real algebraic curve, then $G$ can be realized 
  as the Poincar\'e--Reeb graph 
  of a simply connected, non-compact  algebraic domain in $(\Rr^2, x)$.
\end{theorem}

\begin{figure}[H]
	\begin{center}
		\small
		\tikzstyle{every picture}=[scale=1.0*1]
		\input{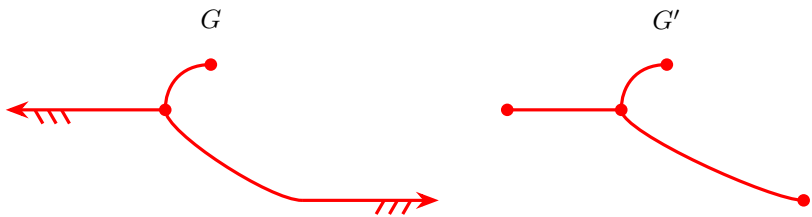}
	\end{center}
    \caption{An example of a tree $G$ (left) and its corresponding compact tree $G'$ (right) after the edges ended with arrows have been replaced by long enough edges having circle vertices.}
    \label{fig:replaceArrowsWithCircleVertices}
\end{figure}

\begin{remark}
Note that in this section we work under the additional hypothesis that the realization from the compact setting is done by a \emph{connected real algebraic curve} and not by a \emph{connected component of a real algebraic curve} as it was done in Theorem \ref{th:realizationA}. We impose this hypothesis, in order not to have difficulties when taking neighboring curves (see Remark \ref{rk:notGoodSituation}).
\end{remark}

\begin{proof}
By hypothesis, there exists a connected real algebraic curve $\mathcal{C}:(f=0)$, $f\in\mathbb{R}[x,y]$ such that $\mathcal{C}$ realizes the newly obtained tree $G'$. In this proof we consider Poincar\'e--Reeb graphs in the source in the sense of Subsection \ref{ssec:PRsource}, so that the graph is situated in the same  plane as the connected real algebraic curve $\mathcal{C}:(f=0)$.

The key idea of the proof is to choose appropriately a non-compact algebraic curve $\mathcal{C}':(g=0)$, $g\in\mathbb{R}[x,y]$ such that when we take a neighboring level of the product of the two polynomials, say $(fg=\epsilon)$ for a sufficiently small $\epsilon>0$, we obtain the desired shape at infinity described by Case A, B or C.
Note that the vertices of the Poincar\'e--Reeb graph are, by definition, transversal intersection points between the polar curve and the level curve. So a small deformation of the level curve will not change this property. Moreover, the neighboring curve must preserve the total preorder between the vertices of the tree. Since there are finitely many such vertices, we can choose $\epsilon$ small enough to ensure this condition holds.

Let us give more details depending on the cases A, B or C.

\textbf{Case A.}
Our goal is to realize the tree from Case A. Namely, we want to add two new non-compact branches that are unbounded in the same direction (see Figure \ref{fig:zoomcaseA}). In order to achieve this, we shall consider the graph $(g=0)$ of a parabola that is tangent to the curve $(f=0)$ in the rightmost vertex of $G'$. Next, consider the real bivariate function $fg:\mathbb{R}^2\rightarrow\mathbb{R}$. The level curve $(fg=0)$ is the union of $\mathcal{C}$ and $\mathcal{C}'.$ Finally, a neighboring curve $(fg=\epsilon)$ realizes the tree $G$, for $\epsilon \neq  0$ sufficiently small.

\begin{figure}[H]
	\begin{center}
		\small
		\tikzstyle{every picture}=[scale=1.0*0.8]
		\input{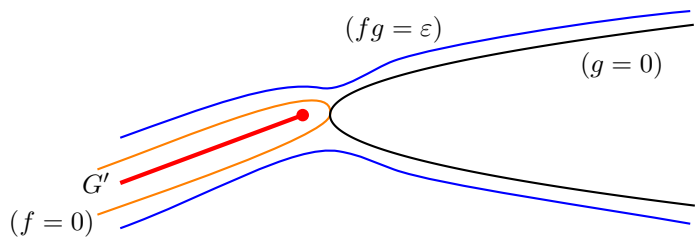}
	\end{center}	
	\caption{Zoom on the construction for case A.\label{fig:zoomcaseA}}
\end{figure}

\begin{example}
Here are the pictures of a graph $G'$ (Figure \ref{fig:excaseA1}) and its realization (Figure \ref{fig:excaseA2}).
\begin{figure}[H]
	\begin{center}
		\small
		\tikzstyle{every picture}=[scale=1.0*0.5]
		\input{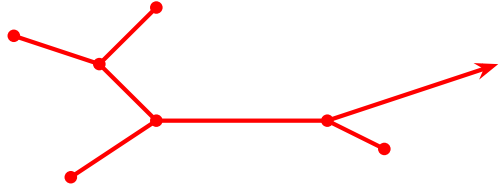}
	\end{center}	
	\caption{Graph $G$ to be realized.\label{fig:excaseA1}}
\end{figure}

\begin{figure}[H]
	\begin{center}
		\small
		\tikzstyle{every picture}=[scale=1.0*0.6]
		\input{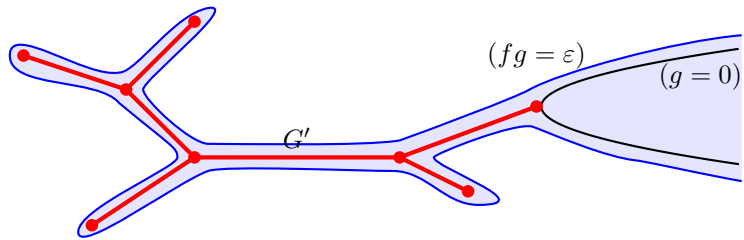}
	\end{center}	
  \caption{Case A: adding two new  branches in the same direction.\label{fig:excaseA2}}
\end{figure}
\end{example}

\medskip

\textbf{Case B.}
In Case B, the goal is to add two new non-compact branches, on opposite sides. First, note that in the presence of two such unbounded branches, the edges decorated by feathers (that is, 
those edges corresponding to the contraction of unbounded segments) form a linear graph $L$.
The extremities of this linear subgraph are the arrowhead vertices of $G$ which we replace by two circular vertices to define $G'$. 

As before, by hypothesis we can consider a connected real algebraic plane curve $\mathcal{C} : (f=0)$ that realizes the graph $G'$. Consider a curve $(g=0)$, algebraic, homeomorphic to a line and situated just below the graph $G'$. More precisely $(g=0)$ is situated in between the linear graph $L$ of $G'$ and the lower part of $(f=0)$ (see Figures \ref{fig:zoomcaseB} and \ref{fig:excaseB2}). The connected component of the neighboring curve $(fg=\epsilon)$ for a sufficiently small $\epsilon \neq 0$ will be the boundary of an algebraic domain that realizes the given tree $G$.

\begin{figure}[H]
	\begin{center}
		\small
		\tikzstyle{every picture}=[scale=1.0*0.8]
		\input{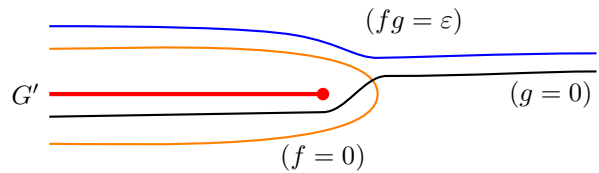}
	\end{center}	
	\caption{Zoom on the construction for case B.\label{fig:zoomcaseB}}
\end{figure}

\begin{example}
Here are the pictures of a graph $G'$ (Figure \ref{fig:excaseB1}) and its realization (Figure \ref{fig:excaseB2}).	
\begin{figure}[H]
	\begin{center}
		\small
		\tikzstyle{every picture}=[scale=1.0*0.5]
		\input{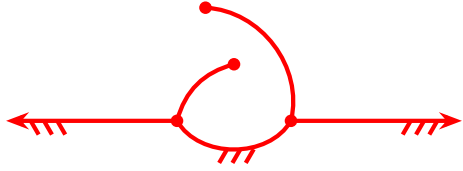}
	\end{center}
	\caption{Graph $G$ to be realized.\label{fig:excaseB1}}
\end{figure}	

\begin{figure}[H]
	\begin{center}
		\small
		\tikzstyle{every picture}=[scale=1.0*0.6]
		\input{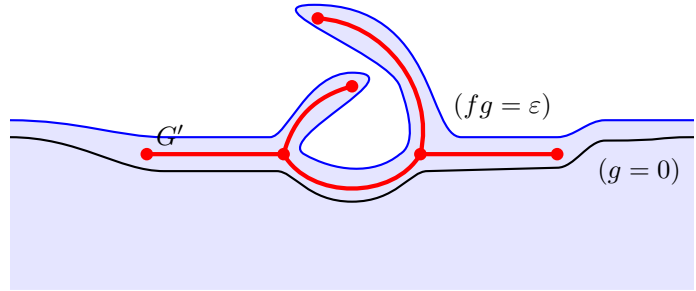}
	\end{center}
  \caption{Case B: adding two new opposite branches. \label{fig:excaseB2}}
\end{figure}

\end{example}

Note that in the above construction there exist other connected components of $(fg=\epsilon)$, for instance in between the curves $(f=0)$ and $(g=0)$, but this is allowed by Definition \ref{def:algebraicDomain}: we considered the algebraic domain $\mathcal{D}$ given by $\partial \mathcal{D}=\mathcal{C},$ where $\mathcal{C} \subsetneq (fg=\epsilon).$

\medskip

\textbf{Case C.}
The domain considered in Case C is the complement of the algebraic domain, say $\mathcal{D}_A$, that we constructed in Case A. Namely, the graph $G$ from Case C is realized by the domain $\mathcal{D}_C$, that is the closure of $\mathbb{R}^2\setminus \mathcal{D}_A.$ Note that in this case the two domains have the same boundary: $\partial \mathcal{D}_C=\partial \mathcal{D}_A=(fg=\epsilon)$ and they are semialgebraic domains.   
\end{proof}

\begin{remark}
\label{rk:notGoodSituation}
Our construction for Theorem \ref{th:realizationB} needs the graph $G'$ to be realized by a connected real algebraic curve. Theorem \ref{th:realizationA} only realizes $G'$ as one connected component $\mathcal{C}_1$ of a real algebraic plane curve $\mathcal{C}$ defined by $(f=0)$; this is not sufficient for our construction. For instance the oval $\mathcal{C}_1$ may be nested inside an oval $\mathcal{C}_2 \subset \mathcal{C}$; the curve $(fg=\epsilon$) of the proof of Theorem \ref{th:realizationB} would no longer satisfy the requested conclusion.

\begin{figure}[H]
	\begin{center}
		\small
		\tikzstyle{every picture}=[scale=1.0*1]
		\input{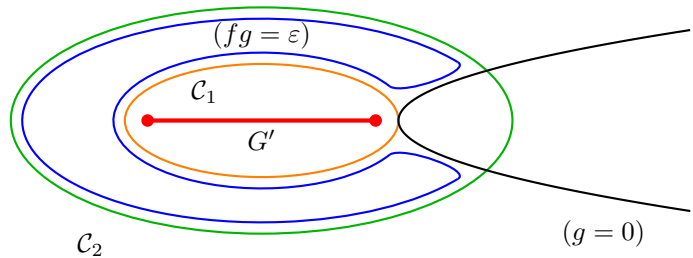}
	\end{center}
	\caption{Construction that does not satisfy the desired conclusion.}
\end{figure}

\end{remark}

\section{General domains of weakly finite type}
\label{sec:general}

We consider the case of $\mathcal{D}$ being any real algebraic domain. Each connected 
component of $\mathcal{C}=\partial\mathcal{D}$ is either an \emph{oval} (a component 
homeomorphic to a circle)
or a \emph{line} (in fact a component homeomorphic to an affine line). An essential question 
in plane real algebraic geometry is to study the relative position of these components. 

\subsection{Combinatorics}      

Let $(\mathcal{P} ,\pi)$ be a vertical plane and a generic domain $\mathcal{D} \subset \mathcal{P}$ of weakly finite type. The next result shows that the Poincar\'e--Reeb graph of $\mathcal{D}$ allows to recover the numbers of lines and ovals of $\mathcal{C} = \partial \mathcal{D}$.

\begin{proposition}\label{prop:bordBetti} 
~
\begin{itemize}
    \item The number of lines in $\mathcal{C}$ is:
       $$\#\{\text{arrows without feathers}\} + \frac12 \#\{\text{arrows with simple feathers}\}.$$
  \item The number of ovals in $\mathcal{C}$ is:
       $$b_0(G)+b_1(G) - c$$
   where $b_0(G)$ is the number of connected components of $G$, $b_1(G)$ is the number 
   of independent cycles in $G$ and $c$ is the number of connected components of 
   $G$ having an arrowhead vertex.
\end{itemize} 
\end{proposition}

\begin{example}
Let us consider Figure \ref{fig:ovals}. One arrowhead without feathers and (half of) two arrowheads with simple feathers, 
give a number of two lines.  As $b_0(G)=3$, $b_1(G)=2$ and $c=2$, we see that 
$b_0(G)+b_1(G) - c =  3$ is indeed the number of ovals in $\mathcal{C}$.

\begin{figure}[H]
	\begin{center}
		\small
		\tikzstyle{every picture}=[scale=1.0*0.9]
		\input{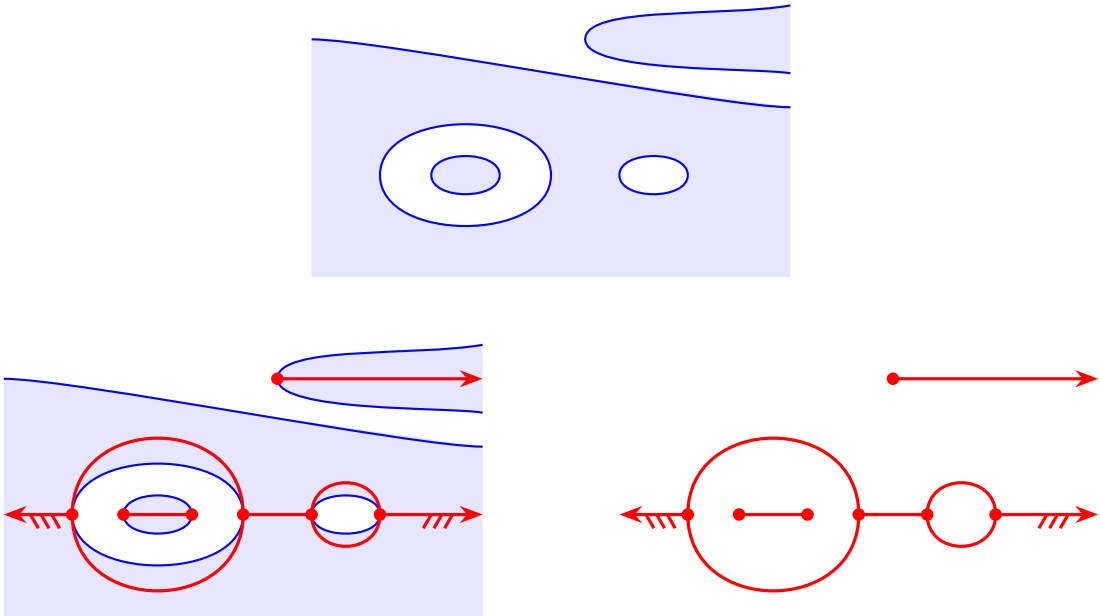}
	\end{center}	
  \caption{Ovals and lines and their Poincar\'e--Reeb graph.\label{fig:ovals}}
\end{figure}

\end{example}

\begin{proof}
For the first point we just notice that each line contributes to either an arrow without feathers or to two arrows with simple feathers.

For the second point, the proof is by induction on the number of ovals.
If there are no ovals, then $b_0(G)=c$, and $b_1(G)=0$, therefore the formula is valid.
Now start with a configuration $\mathcal{C} = \partial \mathcal{D}$ and add an oval 
that does not contain any other ovals. Let $\mathcal{C}'$ be the new curve and $G'$ its graph. 
Either the interior of the new oval is in $\mathcal{D}$, in which case $b_0(G')=b_0(G)$ 
and $b_1(G')=b_1(G)+1$, or the interior of the new oval is in 
$\mathcal{P} \setminus \mathcal{D}$, in which case $b_0(G')=b_0(G)+1$ and 
$b_1(G')=b_1(G)$. In both cases $c(G')=c(G)$. 
Conclusion: $b_0(G')+b_1(G') - c= (b_0(G)+b_1(G) - c) + 1$.
\end{proof}

\subsection{Interior and exterior graphs of domains of weakly finite type}    

Let $\mathcal{D}$ be a generic domain of weakly finite type in a vertical plane $(\mathcal{P}, \pi)$. 
Then the closure $\mathcal{D}^c$ of $\mathcal{P} \setminus \mathcal{D}$ in $\mathcal{P}$ 
is again a domain of weakly finite type, as $\partial \mathcal{D} = \partial \mathcal{D}^c$. 
We say that the Poincar\'e--Reeb graph $G$ of $\mathcal{D}$ is the \emph{interior graph} 
of $\mathcal{D}$ 
and that the Poincar\'e--Reeb graph $G^c$ of $\mathcal{D}^c$ is the \emph{exterior graph} 
of $\mathcal{D}$. 

\begin{figure}[H]
	\begin{center}
		\small
		\tikzstyle{every picture}=[scale=1.0*0.8]
		\input{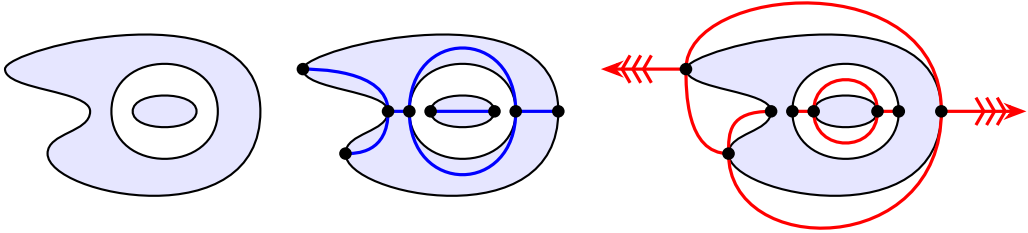}
	\end{center}	
  \caption{Interior and exterior graphs of a domain of weakly finite type.}
\end{figure}

In the next proposition, Poincar\'e-Reeb graphs are to be considered in the sense of 
Definition \ref{def:domainOfWeaklyFiniteType}, that is, as Poincar\'e-Reeb graphs 
in the source:

\begin{proposition}
\label{prop:intExt}
The interior graph $G$ of a domain $\mathcal{D}$ of weakly finite type determines its exterior graph $G^c$.
\end{proposition}

\begin{proof}
The two graphs share the same non-arrowhead vertices.
The local situation around a  non-arrowhead vertex is in accordance to the \emph{trident rule}, where an exterior vertex is replaced by an interior vertex and vice-versa (see Figure \ref{fig:trident}). We also extend this rule to arrowhead vertices. 

\begin{figure}[H]
	\begin{center}
		\small
		\tikzstyle{every picture}=[scale=1.0*0.8]
		\input{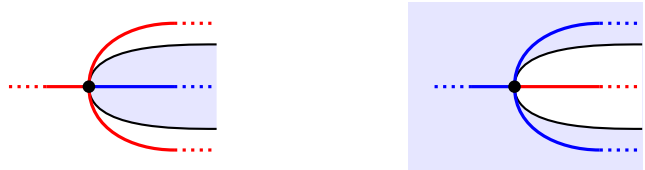}
	\end{center}	
  \caption{The trident rule.\label{fig:trident}}
\end{figure}

Now we derive $G^c$ from $G$ in two steps.

\emph{First step:} make a local construction of the beginning of the edges of $G^c$ according to the trident rule (see Figure \ref{fig:starttrident}).

\begin{figure}[H]
	\begin{center}
		\small
		\tikzstyle{every picture}=[scale=1.0*0.8]
		\input{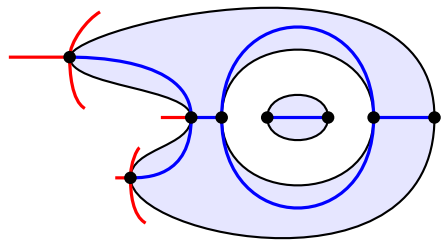}
	\end{center}	
  \caption{The trident rule applied at some vertices (here three vertices are completed).\label{fig:starttrident}}
\end{figure}

\emph{Second step:} complete each edge. It can be done in only one way 
up to vertical isotopies (see for instance Figure \ref{fig:complete}).

\begin{figure}[H]
	\begin{center}
		\small
		\tikzstyle{every picture}=[scale=1.0*0.8]
		\input{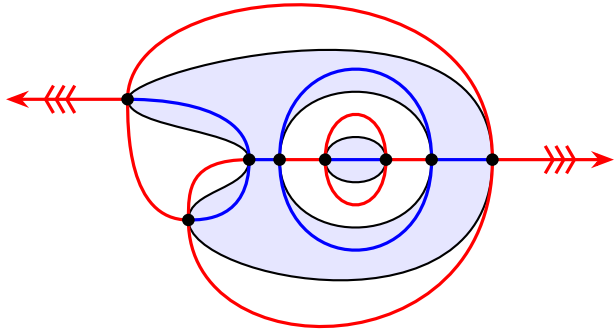}
	\end{center}	
  \caption{Completed exterior graph.\label{fig:complete}}
\end{figure}

\end{proof}


\bibliographystyle{amsplain}
\bibliography{references.bib}

\providecommand{\bysame}{\leavevmode\hbox to3em{\hrulefill}\thinspace}
\providecommand{\MR}{\relax\ifhmode\unskip\space\fi MR }
\providecommand{\MRhref}[2]{%
  \href{http://www.ams.org/mathscinet-getitem?mr=#1}{#2}
}
\providecommand{\href}[2]{#2}
\begin{thebibliography}{10}

\bibitem{bagby}
Thomas Bagby, Leonard~Peter Bos, and Norman~Jr. Levenberg, \emph{Multivariate
  simultaneous approximation}, Constr. Approx. \textbf{18} (2002), no.~4,
  569--577.

\bibitem{MR4472570}
Saugata Basu, Nathanael Cox, and Sarah Percival, \emph{On the {R}eeb spaces of
  definable maps}, Discrete Comput. Geom. \textbf{68} (2022), no.~2, 372--405.

\bibitem{DeLellis2019}
Camillo De~Lellis, \emph{The {M}asterpieces of {J}ohn {F}orbes {N}ash {J}r.},
  The {A}bel {P}rize 2013-2017 (Helge Holden and Ragni Piene, eds.), Springer
  International Publishing, Cham, 2019, pp.~391--499.

\bibitem{MR3505333}
Vin de~Silva, Elizabeth Munch, and Amit Patel, \emph{Categorified {R}eeb
  graphs}, Discrete Comput. Geom. \textbf{55} (2016), no.~4, 854--906.

\bibitem{MR3685708}
Tamal~K. Dey, Facundo M\'{e}moli, and Yusu Wang, \emph{Topological analysis of
  nerves, {R}eeb spaces, mappers, and multiscale mappers}, 33rd {I}nternational
  {S}ymposium on {C}omputational {G}eometry, LIPIcs. Leibniz Int. Proc.
  Inform., vol.~77, Schloss Dagstuhl. Leibniz-Zent. Inform., Wadern, 2017,
  pp.~Art. No. 36, 16.

\bibitem{MR2504290}
Herbert Edelsbrunner, John Harer, and Amit~K. Patel, \emph{Reeb spaces of
  piecewise linear mappings}, Computational geometry ({SCG}'08), ACM, New York,
  2008, pp.~242--250.

\bibitem{Elredge}
Nate Elredge, \emph{Answer to \emph{On finding polynomials that approximate a
  function and its derivative}}, StackExchange, question 555712 (2013).

\bibitem{Gh1}
\'Etienne Ghys, \emph{A singular mathematical promenade}, ENS \'Editions, Lyon,
  2017.

\bibitem{MR1867354}
Allen Hatcher, \emph{Algebraic topology}, Cambridge University Press,
  Cambridge, 2002.

\bibitem{MR4476514}
Naoki Kitazawa, \emph{On {R}eeb graphs induced from smooth functions on
  3-dimensional closed manifolds with finitely many singular values}, Topol.
  Methods Nonlinear Anal. \textbf{59} (2022), no.~2B, 897--912.

\bibitem{MR4548150}
\bysame, \emph{On {R}eeb graphs induced from smooth functions on closed or open
  manifolds}, Methods Funct. Anal. Topology \textbf{28} (2022), no.~2,
  127--143.

\bibitem{MR3850451}
Jussi Klemel\"{a}, \emph{Level set tree methods}, Wiley Interdiscip. Rev.
  Comput. Stat. \textbf{10} (2018), no.~5, e1436, 14.

\bibitem{lerarioStecconi}
Antonio Lerario and Michele Stecconi, \emph{What is the degree of a smooth
  hypersurface?}, J. Singul. \textbf{23} (2021), 205--235.

\bibitem{MR2828379}
Yasutaka Masumoto and Osamu Saeki, \emph{A smooth function on a manifold with
  given {R}eeb graph}, Kyushu J. Math. \textbf{65} (2011), no.~1, 75--84.

\bibitem{Po}
Henri Poincar\'e, \emph{Papers on topology, cinqui\`eme compl\'ement à
  l'analysis situs}, History of Mathematics, vol.~37, American Mathematical
  Society, Providence, RI; London Mathematical Society, London, 2010,
  Rendiconti del Circolo Matematico di Palermo (1884-1940), 45--110, Springer,
  translated and with an introduction by John Stillwell, 1904.

\bibitem{Re}
Georges Reeb, \emph{Sur les points singuliers d'une forme de {P}faff
  compl\`etement int{\'e}grable ou d'une fonction num{\'e}rique}, C. R. Acad.
  Sci. Paris \textbf{222} (1946), 847--849.

\bibitem{MR3822880}
Osamu Saeki, \emph{Theory of singular fibers and {R}eeb spaces for
  visualization}, Topological methods in data analysis and visualization. {IV},
  Math. Vis., Springer, Cham, 2017, pp.~3--33.

\bibitem{saeki2022}
\bysame, \emph{Reeb spaces of smooth functions on manifolds}, Int. Math. Res.
  Not. IMRN (2022), no.~11, 8740--8768.

\bibitem{MR87106}
Stephen Smale, \emph{A {V}ietoris mapping theorem for homotopy}, Proc. Amer.
  Math. Soc. \textbf{8} (1957), 604--610.

\bibitem{sorea2018shapes}
Miruna-\c{S}tefana Sorea, \emph{The shapes of level curves of real polynomials
  near strict local minima}, Ph.D. thesis, Universit{\'e} de Lille/Laboratoire
  Paul Painlev{\'e}, 2018.

\bibitem{sorea2019constructing}
\bysame, \emph{Constructing separable {A}rnold snakes of {M}orse polynomials},
  Port. Math. \textbf{77} (2020), no.~2, 219--260.

\bibitem{sorea2020measuring}
\bysame, \emph{Measuring the local non-convexity of real algebraic curves}, J.
  Symbolic Comput. \textbf{109} (2022), 482--509.

\bibitem{sorea2019permutations}
\bysame, \emph{Permutations encoding the local shape of level curves of real
  polynomials via generic projections}, Ann. Inst. Fourier (Grenoble)
  \textbf{72} (2022), no.~4, 1661--1703.

\bibitem{Stone}
Marshall~Harvey Stone, \emph{The generalized {W}eierstrass approximation
  theorem}, Math. Mag. \textbf{21} (1948), 167--184, 237--254.

\end{thebibliography}

\end{document}